\documentclass[11pt,twoside,final]{scrartcl}
\usepackage{url,a4wide}
\usepackage{amsmath, amsfonts, amssymb, amsthm, mathtools, nicefrac}
\usepackage{todonotes}
\usepackage{enumitem}
\usepackage{graphicx}
\graphicspath{{pics/}}
\usepackage[skip=1pt, font=normalsize]{subcaption}
\usepackage{hyperref} % provides \url command for bibtex and links to jump within documents
\hypersetup{plainpages=false, colorlinks, linkcolor=black, citecolor=black, urlcolor=blue,
pdftitle={Approximation by short exponential sums with geometric error decay based on Gauss quadrature},
pdfauthor={Gerlind Plonka, Yannick Riebe, Annie Cuyt}}
\allowdisplaybreaks

\newcommand{\diag}{\mathrm{diag}}

\newcommand{\N}{\mathbb N}

\def\sumprime_#1^#2{
    \setbox0=\hbox{$\scriptstyle{#1}$}
    \setbox1=\hbox{$\scriptstyle{#2}$}
    \setbox2=\hbox{$\textstyle{\sum}$}
    \setbox4=\hbox{${}^\prime\mathsurround=0pt$}
    \dimen0=.5\wd0 \advance\dimen0 by-.5\wd2
    \ifdim\dimen0>0pt
        \ifdim\dimen0>\wd4 \kern\wd4
        \else\kern\dimen0
        \ifdim\dimen1>\wd4 \kern\wd4
        \else\kern\dimen1
    \fi\fi\fi
\mathop{{\sum}^\prime}_{\kern-\wd4 #1}^{\kern-\wd4 #2}
}

\newtheorem{theorem}{Theorem}[section]
\newtheorem{corollary}[theorem]{Corollary}
\newtheorem{lemma}[theorem]{Lemma}
\newtheorem{alg}[theorem]{Algorithm}
\newtheorem{definition}[theorem]{Definition}
\newtheorem{example}[theorem]{Example}
\newtheorem{remark}[theorem]{Remark}

\newcommand{\bend}{\hspace*{0ex} \hfill \hbox{\vrule height
	1.5ex\vbox{\hrule width 1.4ex \vskip 1.4ex\hrule  width 1.4ex}\vrule
	height 1.5ex}}

\newenvironment{Corollary}{\goodbreak\begin{corollary}\upshape}{\end{corollary}}
\newenvironment{Example}{\goodbreak\begin{example}\upshape}{\bend\end{example}}
\newenvironment{algorithm}[1]{\goodbreak~\begin{alg}[#1]~\vspace{-9pt}~\\
		\rule{\linewidth}{0.5pt}~\\}{\vspace{-9pt}~\\
		\rule{\linewidth}{0.5pt}~\end{alg}}

\numberwithin{equation}{section}
\numberwithin{table}{section}
\numberwithin{figure}{section}

\renewcommand{\mathbf}[1]{\ensuremath{\boldsymbol{#1}}}

\title{Approximation by short exponential sums with geometric error decay based on Gauss quadrature}
\author{
	Gerlind Plonka\footnotemark[1], 
    Yannick Riebe\footnotemark[2], and Annie Cuyt\footnotemark[3]
}

\date{}
\begin{document}
\maketitle

\begin{abstract}
\textbf{Abstract.} We present new short explicit exponential sum approximations of length 
$N$ for $f_1(x)=\frac{1}{a+x}$ with $a>0$ on $[0, \infty)$ and for $f_2(x)= {\mathrm e}^{-x^2/(2\sigma)}$ with $\sigma>0$ on ${\mathbb R}$
with geometric error decay ${\rho}^{-2N}$
  for user-defined $N \ge 2$ and $\rho >1$.
The approximations are built over consecutive intervals $[b_j, \, b_{j+1}) \subset [0, \infty)$, $j \in {\mathbb N}_{0}$,
with interval lengths that depend on $\rho$ and grow exponentially for $f_1$
and are equidistant for $f_2$.
All parameters determining the exponential sum approximations on $[b_j, \, b_{j+1})$
 are explicitly given and are closely related to initial parameters on $[b_0, \, b_{1})$.
Our method is based on Gauss-Laguerre
and Gauss-Hermite
quadrature, respectively, applied to suitable parametric integral representations of $f_1$ and $f_2$.
This technique ensures consistent relative errors across all intervals.
 Since all parameters of the exponential sum are given explicitly in closed form, 
in terms of $N$, $\rho$, and the nodes and weights of the underlying Gauss quadrature rule, 
our method avoids the non-linear optimization and the numerical instabilities inherent in 
Prony-type or Remez-based constructions. 
Using the obtained exponential 
sum approximations, we achieve highly accurate approximations of 
$\log(x)$ on $[1,\infty)$ and of the error function $\mathrm{erf}(x)$ with predictable geometric error decay.
Numerical examples for $N=8$ and $N=10$ clearly illustrate the theoretical error estimates.
\medskip

	\smallskip

\noindent	
	AMS \emph{Subject Classifications}:
	41A25, 41A30, 41A55,  65D15, 65D32
\end{abstract}

\footnotetext[1]{Corresponding author: plonka@math.uni-goettingen.de, University of G\"{o}ttingen, Institute for Numerical and Applied Mathematics, D--37083 G\"{o}ttingen, Germany}
\footnotetext[2]{y.riebe@math.uni-goettingen.de, University of G\"{o}ttingen, Institute for Numerical and Applied Mathematics, D--37083 G\"{o}ttingen, Germany}
\footnotetext[3]{annie.cuyt@stir.ac.uk, University of Stirling, Computing Science and Mathematics - Division, DFK9 4LA
Scotland UK}

\section{Introduction}
Exponential sum approximations of smooth functions are of great interest in various applications since they can easily be integrated, differentiated, convolved, and applied in operator form. Consequently, these approximations play 
an important role in the representation of kernels and Green's functions.

In this paper, we consider the approximation of two special analytic functions by short exponential sums.
The first function $f_1(x) := \frac{1}{a+x}$ with $a \in {\mathbb R}$, $a >0$, is considered on $[0, \infty)$. 
 Exponential sum approximations of $f_1(x)$ are of particular interest for modelling fractional diffusion operators with $1/x$-type kernels, numerical evaluation of high-dimensional integrals in quantum chemistry, and fast multipole methods, see \cite{Beylkin05, Braess05, Beylkin10, Hackbusch19, Braess26, Bilokon26, Bilokon26a, Gim20}.

The second function is the Gaussian $f_2(x) := {\mathrm e}^{-x^2/(2\sigma)}$  on ${\mathbb R}$,  with $\sigma >0$ playing the role of the variance.
The Gaussian, which plays an important role in statistics, signal processing, molecular modeling, computational chemistry, as well as in approximation theory, see e.g. \cite{Bachmayr14, Jiang22, Derevianko26}.
\medskip

\textbf{Problem statement and main results.}
Our goal is to find exponential sum approximations of length $N$ for $f(x)= f_1(x)$ and $f(x)=f_2(x)$ with a
 geometric error decay rate $\rho^{-2N}$, where we can choose $N\ge 2$ and $\rho>1$ arbitrarily in advance, and where all required 
parameters determining the exponential sum  are explicitly given based on zeros of orthogonal polynomials and
can be easily computed with high accuracy.  In particular, we are interested in a rather short exponential sum of length $N$, 
while the approximation error can be made arbitrarily small by choosing a large $\rho$.
 
Obviously, this goal cannot be achieved for all $x \in [0, \infty)$ simultaneously. Instead, we propose to adapt the exponential sum approximation over consecutive intervals $[b_j, \, b_{j+1})$ with $b_0=0$ and $b_{j+1} > b_j$, $j \in {\mathbb N}_0$, 
where the length of the intervals $b_{j+1}-b_j$ depends on the desired decay parameter $\rho$. The achieved approximations satisfy 
$$  \left|f(x) - \sum\limits_{k=1}^N w_k^{(j)} \, {\mathrm e}^{t_k^{(j)} x} \right| \le  c^{(j)} \rho^{-2N},  \qquad x \in [b_j, b_{j+1}), \qquad j \in \N_0,$$
where the sequence $(c^{(j)})_{j=0}^{\infty}$ decays exponentially with $c^{(j+1)}< c^{(j)} < c^{(0)}=\frac{2}{a} \frac{\rho}{\rho-1}$ 
for $f=f_1$ and $c^{(j+1)}< c^{(j)} < c^{(0)} =\sqrt{\frac{2N+1}{4N}}$
for $f=f_2$, $j\in\N$. Further, 
 we require that the parameters $w_k^{(j)}$, $t_k^{(j)}$
for  $j >0$ can be easily  represented using
$w_k^{(0)}$, $t_k^{(0)}$.  
 We will be able to take an
 exponentially growing interval length $b_{j+1}-b_j=\frac{2a}{\rho-1}\big(\frac{\rho+1}{\rho-1}\big)^j$ for $f_1$
 and an equidistant interval length  $b_{j+1}-b_j=\frac{2}{\rho} \sqrt{\frac{N \sigma}{{\mathrm e}}}$ for $f_2$.

Our construction of exponential sum approximations is strongly based on the observation that the functions $f_1$ and $f_2$ admit parametric integral representations that allow the direct application of Gauss quadrature formulas.
 The function $f_1(x) = \frac{1}{a+x}$ with $x \ge 0$ and $a>0$ can be represented as a parameter integral of the form 
\begin{align}\label{rep11}
\textstyle f_{1}(x) =\frac{1}{x+a} = \frac{1}{a b} \int\limits_0^{\infty}  {\mathrm e}^{-s\big(\frac{x+a}{ab}-1\big)} \, {\mathrm e}^{-s} \, {\mathrm d}s, 
\end{align}
where $b>0$ is an arbitrary parameter.
The function $f_{2}(x) = {\mathrm e}^{-x^2/(2\sigma)}$
admits the representation
\begin{align}\label{rep2} \textstyle f_{2}(x) = {\mathrm e}^{-\frac{x^2}{2\sigma}} =  \frac{1}{\sqrt{2\pi}} \int\limits_{-\infty}^{\infty} {\mathrm e}^{\frac{-s^2}{2}} \, 
{\mathrm e}^{-{\mathrm i} \frac{x}{\sqrt{\sigma}}s}\, {\mathrm d}s
= \frac{1}{\sqrt{\pi}}  \int\limits_{-\infty}^{\infty}  {\mathrm e}^{-u^2} \, \cos\big( \frac{x}{\sqrt{\sigma}} \sqrt{2} u\big)  \, {\mathrm d} u.
\end{align}
Applying Gauss-Laguerre quadrature to $f_{1}$ and Gauss-Hermite quadrature to the Gaussian $f_{2}$ directly yields exponential sum approximations with geometric error decay for small $|x|$ in both instances. 
To achieve similarly good approximations for larger $|x|$, we furthermore develop a strategy to employ structurally similar exponential sum approximations on consecutive intervals, such that the relative error of our exponential sum 
approximation remains the same for every considered interval.
While the obtained approximations are not optimal in special norms, they satisfy the desired error 
bound 
and can be easily computed without 
any numerical stability issues. 

We also show that the obtained approximations can be directly applied to achieve a highly accurate exponential sum approximation of the function $\log(x)$ on $[1, \infty)$ and an approximation of the error function by a short sum of Gaussians, both with predetermined geometric error decay.

\medskip

\textbf{Related Literature.}
There have been several different approaches to compute exponential sum approximations for $f_1(x)$, see e.g.\ \cite{Braess86, Braess95, Braess05, Beylkin05, Beylkin10, Potts13, McLean18, Hackbusch19, Braess26, Gim20}, as well as $f_2(x)$, see 
\cite{Schaback79, Jiang22, Derevianko26}.
These techniques are often based on numerical realizations of Prony's method, where the nonlinear problem of minimizing
$\| f(x) - \sum\limits_{k=1}^N w_k {\mathrm e}^{T_k x}\|$ with respect to $w_k$ and $T_k$ is replaced by a discrete (interpolation) problem. Several modifications of the classical Prony method have been proposed 
 to achieve higher numerical stability, see \cite{Beylkin05,Beylkin10,Potts13,Plonka14,Derevianko23,Zhang19,Briani20,Cuyt20}
and \cite{Plonka23}, chapter 10.  Also Bessel functions can be efficiently approximated  
by these methods, see \cite{Beylkin05,Cuyt20,Derevianko23,Razavi23}.

In some papers the close connection to rational approximation is exploited \cite{Jiang22, Derevianko23, Razavi23, Braess26, Kuznetsov26}, since the Laplace transform of an exponential sum of length $N$ yields a rational function of type $(N-1,N)$.
The approach in \cite{Hackbusch19} generalizes the Remez algorithm to achieve best approximations of $\frac{1}{x}$ in the $L^{\infty}$-norm.
\smallskip

Suboptimal exponential approximation methods in \cite{Schaback79,Kammler81,Stampfer20,Derevianko26} exploit the fact that an exponential sum can be interpreted as the solution of a homogeneous linear differential equation of order $N$ with constant coefficients.
Assuming that $f$ can be well approximated by an exponential sum of length $N$, one can try to find a linear differential operator $D_{N}$ of order $N$ that minimizes $\| D_N f\|$ in some norm and yields suitable frequency parameters for the exponential sum approximation.

The use of quadrature formulas to derive suboptimal exponential sum approximations
from integral representations of $f$
has been proposed e.g.\ in \cite{Beylkin05,Beylkin10,McLean18,Jiang22,Braess26,Bilokon26,Bilokon26a,Gim20}.  In particular, 
the trapezoidal rule has been applied in \cite{Beylkin05,Beylkin10} and in \cite{Braess26}. The Gauss-Hermite quadrature rule appears in \cite{Derevianko26}, where it is however not directly used to compute an exponential sum approximation of the Gaussian. Very recently, in \cite{Bilokon26} the application of 
Gauss-Laguerre quadrature for kernel approximation has been proposed, but without further analysis of approximation errors.

\smallskip

The mentioned numerical methods based on discretized interpolation or approximation mostly lack theoretical error estimates. For the function $f_1(x)$, which is a special instance of a completely monotonic function, there exist investigations into the error of optimal approximation by exponential sums of fixed order $N$ on compact
 intervals $[a,b] \subset [0, \infty)$ and on $[0, \infty)$ in the $L^{\infty}$-norm and in weighted norms, see \cite{Braess86,Braess95,Braess05,Braess26}.
In \cite{Tasche26} it has been shown that $\frac{1}{x}$ can be approximated on an interval $[a,b]$ 
(with $0 < a < b < \infty$) by algebraic polynomials of degree $N$ leading to geometric error decay of the form $\rho^{N+1}$, where $\rho= v- \sqrt{v^2-1}$ with $v=\frac{a+b}{b-a}$.
Regarding error estimates for $f_2(x)$, we are only aware of the results in \cite{Derevianko26}, where 
the Gaussian is approximated by a short cosine sum, and error estimates in the $L^2$-norm or the weighted $L^2$-norm employ the Gauss-Hermite quadrature.

Our approach is significantly different from the known exponential sum approximations considered so far, since we can fix the desired error bound as well as the length of the sum in advance, and can determine all required parameters of the exponential sum based on the weights and nodes of the corresponding Gauss-type quadrature formulas.
The nodes and weights for the required Gauss quadrature formulas can be computed with high accuracy, see \cite{Townsend16, Gil19}. Moreover, we will show that exponential sums of order up to $10$ are sufficient to achieve errors in the range $10^{-15}$ in double-precision arithmetic. 
While all approximations for the function in question are meant to be programmed in double precision, some graphs were made using the higher precision arithmetic package \texttt{vpa} when the relative error dropped below the value $2^{-53}$ which is the double precision round-off threshold for the basic operations and square root.
Furthermore, since $N$ and $\rho$ can be chosen freely, 
the same construction directly caters for target accuracies below or above the usual 
double-precision range, simply by adjusting these two parameters.
\medskip

\textbf{Outline of the paper.}
In Section 2, we consider the approximation of $f_1(x)= \frac{1}{x+a}$ by short exponential sums, based on the 
application of Gauss-Laguerre quadrature to a parametric integral representing $f_1$. To achieve the desired approximation result, we first derive a new error estimate for Gauss-Laguerre quadrature applied to the 
special function $g(s)={\mathrm e}^{-s(\frac{x+a}{ab}-1)}$, where $x$, $a$ and $b$ are taken as positive parameters, see Theorem \ref{theoer}. Based on this error estimate we derive our first main Theorem \ref{theo2b}, which contains the exponential sum approximations for $f_1(x)$ and the corresponding error estimate.
In Section 2.2 we also compute the relative error of this approximation, which has the same structure and magnitude in every considered interval. Further, we give numerical examples for the highly accurate approximation of $f_1(x)$ with exponential sums of
 lengths $8$ and $10$. Section 2.3 provides an application of the obtained approximation, namely the derivation of geometrically decaying exponential sum approximations of $\log(x)$. 
 \smallskip
 
In Section 3, we study the approximation of $f_2(x)={\mathrm e}^{-x^2/(2\sigma)}$ by short exponential sums, based on the Gauss-Hermite quadrature rule applied to a suitable parametric integral representation of $f_2$. The second main Theorem \ref{Happrox} presents the new exponential sum approximation of $f_2$ together with the error estimate providing a geometric error decay. In Section 3.2 we consider the computational effort of the exponential sum approximation, estimate the relative error (which has the same structure and magnitude in every considered interval), and give numerical examples for $N=8$ and $N=10$. Section 3.3 then presents an application to approximate the error function $\mathrm{erf}(x)$ with high accuracy. 

All numerical computations mostly use double-precision arithmetic, and our numerical results show that the achieved errors can readily reach the range of $10^{-15}$.

\section{Approximation of $\frac{1}{a+x}$ on $[0, \infty)$ by short exponential sums}

%Observe that the function $f_1(x) = \frac{1}{a+x}$ with $x \ge 0$ and $a>0$ can be represented as a parameter integral of the form 
%\begin{align}\label{rep11}
%\textstyle \frac{1}{x+a} = \frac{1}{a b} \int\limits_0^{\infty}  {\mathrm e}^{-s\big(\frac{x+a}{ab}-1\big)} \, {\mathrm e}^{-s} \, {\mathrm d}s, 
%\end{align}
%where $b>0$ is an arbitrary parameter.

\subsection{Approximation based on Gauss-Laguerre Quadrature}
Our goal is to find an approximation of $f_1(x)$ in (\ref{rep11}) by a short exponential sum with geometric error decay, 
using the Gauss-Laguerre quadrature. For this purpose, we first briefly recall the known results on this quadrature rule and derive a new error estimate for the function $\tilde{g}_{b,x} (s) := \frac{1}{ab} \, {\mathrm e}^{-s(\frac{x+a}{ab} - 1)}$. Based on this result we will derive our exponential sum approximation of $f_{1}$ in Theorem \ref{theo2b}.

\subsubsection{Error of Gauss-Laguerre quadrature}
Laguerre polynomials $L_{n}(s)$, $n \in {\mathbb N}_{0}$, are determined by their generating function of the form 
\begin{align}\label{gene}
\textstyle \frac{1}{1-z} \, {\mathrm e}^{-\frac{sz}{1-z}} = \sum\limits_{n=0}^{\infty} z^n \, L_n(s), 
\end{align}
see \cite{Szego}, formula (5.1.9). We further recall that the sequence $\{ L_n(t)\}_{n=0}^{\infty}$ of Laguerre polynomials forms an orthogonal system 
on the weighted Hilbert space $L_2([0, \infty), {\mathrm e}^{-t})$ with 
$$ \textstyle \int\limits_0^{\infty} L_n(t) \, L_m(t) \, {\mathrm e}^{-t} \, {\mathrm d}t = \delta_{n,m}, $$
where $\delta_{n,m}$ denotes the Kronecker symbol.

Generally, 
the Gauss-Laguerre quadrature rule is given by 
\begin{align}\label{GL} \textstyle \int\limits_{0}^{\infty} g(s) \, {\mathrm e}^{-s} \, {\mathrm d}s = 
\sum\limits_{k=1}^N \, w_{N,k}^{(L)} \, g(t_{N,k}^{(L)}) + R_N(g), 
\end{align}
where $t_{N,k}^{(L)}$, $k=1, \ldots, N$, denote the positive pairwise distinct  roots of the Laguerre polynomial $L_N(t)$,
and the weights $w_{N,k}^{(L)}$ are given by 
\begin{align}\label{wLaguerre}  \textstyle w_{N,k}^{(L)} := \frac{t_{N,k}^{(L)}}{(N+1)^2 \, [L_{N+1}(t_{N,k}^{(L)})]^2}, \qquad k=1, \ldots , N. 
\end{align}
Based on Peano kernel analysis, $R_N(g)$ can be represented, for smooth functions $g \in C^{2N}([0, \infty))$,
in the form 
\begin{align}\label{GLerror} \textstyle R_N(g) = \frac{(N!)^2}{(2N)!} g^{(2N)}(\xi), 
\end{align}
where $g^{(2N)}(\xi)$ denotes the $(2N)$-th derivative of $g$ at a suitable value $\xi \in (0, \infty)$, see 
\cite{Abramowitz}, Section 25.4.45.
 Application of the quadrature rule (\ref{GL}) to our representation of $f_1(x)$ in (\ref{rep11}), with 
 $\tilde{g}_{b,x}(s):=  \frac{1}{ab}{\mathrm e}^{-s \big(\frac{x+a}{ab}-1\big)}$, leads to 
\begin{align*}
\textstyle \frac{1}{a+x} 
&  \textstyle =  \frac{1}{a b} \int\limits_0^{\infty}  {\mathrm e}^{-s\big(\frac{x+a}{ab}-1\big)} \, {\mathrm e}^{-s} \, {\mathrm d}s  
 = \textstyle  \sum\limits_{k=1}^N \frac{w_{N,k}^{(L)}}{ab} \,  
 {\mathrm e}^{-t_{N,k}^{(L)} \big(\frac{x+a}{ab}-1\big)}  + R_N(\tilde{g}_{b,x})
\end{align*}
with $R_N(\tilde{g}_{b,x}) = \frac{(N!)^2}{(2N)!} \tilde{g}_{b,x}^{(2N)}(\xi_x)$ for some suitable $\xi_x \ge 0$, i.e., 
\begin{align} \label{err0}
|R_N(\tilde{g}_{b,x})| &  \textstyle  
\le \frac{(N!)^2}{(2N)!}  \frac{1}{ab} \big(\frac{x+a}{ab}-1\big)^{2N} < \frac{\sqrt{\pi (N+\frac{1}{2})}}{4^N \, ab} 
\big(\frac{x+a}{ab}-1\big)^{2N},
\end{align}
where we have used Stirling's formula and  $\|\tilde{g}_{b,x}^{(2N)}\|_{\infty} =\frac{1}{ab} \big(\frac{x+a}{ab}-1\big)^{2N}$ for $\big(\frac{x+a}{ab}-1\big) \ge 0$, i.e.,
$x \ge a(b-1)$.
Based on this error estimate, $|R_N(\tilde{g}_{b,x})|$ decays geometrically with $N$ if 
$ \frac{1}{2} (\frac{x+a}{ab}-1\big) <1$, i.e., if $1 \le \frac{x+a}{ab}< 3$, which can be easily achieved by choosing $b$ suitably. Moreover, an error decay of at least $\rho^{-2N}$ is obtained for 
$0 \le \frac{x+a}{ab} -1 \le \frac{2}{\rho}$,
i.e., for 
$$  \textstyle  a(b-1)\le x \le  ab \Big(\frac{2}{\rho} +1\Big) -a = \frac{2ab}{\rho} + a(b-1). $$
Thus, we can fix the order $N$ of the exponential sum and the decay rate $\rho>1$ to approximate $\frac{1}{x+a}$ with a geometrically decaying error $\rho^{-2N}$ on the interval $[a(b-1), \frac{2ab}{\rho} + a(b-1)]$. For example, taking $b=1$ we obtain this decay rate for $x \in [0, \frac{2a}{\rho}]$, while for larger $b$ we can accurately approximate $f_1(x)$ on other intervals. 

However, the error estimate (\ref{err0}) for $|R_N(\tilde{g}_{b,x})|$ is not very sharp and strongly overestimates the error when applied to the entire function $\tilde{g}_{b,x}$.

Therefore, we first derive a new error estimate tailored to our specific purpose.
Afterwards, we show how to approximate $
f_1(x) = \frac{1}{a+x}$ for arbitrary $x$ with high accuracy.

\begin{theorem}\label{theoer}
The Gauss-Laguerre quadrature rule possesses  for $\tilde{g}_{b,x}(s)= \frac{1}{ab}{\mathrm e}^{-s \big(\frac{x+a}{ab}-1\big)}$
an error of the form 
\begin{align}
\nonumber
\textstyle  R_N(\tilde{g}_{b,x}) &= \textstyle   \int\limits_0^{\infty} \tilde{g}_{b,x}(s) \, {\mathrm e}^{-s} \, {\mathrm d} s
- \sum\limits_{k=1}^N w_{N,k}^{(L)} \,   \tilde{g}_{b,x}(t_{N,k}^{(L)})  =  \frac{1}{x+a}
- \sum\limits_{k=1}^N w_{N,k}^{(L)} \,  \tilde{g}_{b,x}(t_{N,k}^{(L)})  \\
\label{errorexact}
& \textstyle =  \textstyle  \frac{-1}{x+a} \big(1-\frac{ab}{x+a}\Big)^{2N} \!\!
 \sum\limits_{n=0}^{\infty} \big( 1-\frac{ab}{x+a}\Big)^n 
\sum\limits_{k=1}^N w_{N,k}^{(L)} \,  L_{n+2N} (t_{N,k}^{(L)}) .
\end{align}
In particular, for $x > a\big(\frac{b}{2}-1\big)$ we have
\begin{align}
\label{err1}
&\textstyle  |R_N(\tilde{g}_{b,x})| \le c \,  \big( 1-\frac{ab}{x+a}\Big)^{2N},
\end{align}
where $c=\frac{2}{ab}$ for $x \ge a(b-1)$ and $c=\frac{2}{2(x+a)-ab}$ for $a\big(\frac{b}{2}-1\big) < x < a(b-1)$.
\end{theorem}
\begin{proof}
Observe that the Laguerre expansion of the function $\tilde{g}_{b,x}(s)=  \frac{1}{ab}
{\mathrm e}^{-s \big(\frac{x+a}{ab}-1\big)}$ is given by
$$ \textstyle \tilde{g}_{b,x}(s)=   \frac{1}{x+a} \sum\limits_{n=0}^{\infty} \big( 1-\frac{ab}{x+a}\Big)^n\, L_n(s). $$
This follows from (\ref{gene}) with $z= 1- \frac{ab}{x+a}$.
Since the Gauss-Laguerre quadrature is exact for polynomials up to degree $2N-1$, we conclude
\begin{align*}
R_N(\tilde{g}_{b,x}) &= \textstyle  \int\limits_0^{\infty} \tilde{g}_{b,x}(s) \, {\mathrm e}^{-s} \, {\mathrm d} s
- \sum\limits_{k=1}^N w_{N,k}^{(L)} \, \tilde{g}_{b,x}(t_{N,k}^{(L)}) \\
&= \textstyle  \frac{1}{x+a} \sum\limits_{n=0}^{\infty} \big( 1-\frac{ab}{x+a}\big)^n\,\Big( \int\limits_0^{\infty}   L_n(s)\, {\mathrm e}^{-s} \, {\mathrm d} s  - \sum\limits_{k=1}^N w_{N,k}^{(L)} \, L_n(t_{N,k}^{(L)})\Big)\\
&= \textstyle  \frac{1}{x+a} \sum\limits_{n=2N}^{\infty} \big( 1-\frac{ab}{x+a}\big)^n\, 
\big( 0 - \sum\limits_{k=1}^N w_{N,k}^{(L)} \, L_n(t_{N,k}^{(L)})\big)\\
&= \textstyle \textstyle  \frac{-1}{x+a} \big( 1-\frac{ab}{x+a}\big)^{2N}
 \sum\limits_{n=0}^{\infty} \big( 1-\frac{ab}{x+a}\big)^n\, 
\sum\limits_{k=1}^N w_{N,k}^{(L)} \, L_{n+2N} (t_{N,k}^{(L)}).
\end{align*}
Applying that $ {\mathrm e}^{-t/2} |L_{n}(t)| \le 1$ for all $t \ge 0$ and all $n \in {\mathbb N}$, see e.g.\ 
\cite{Szego}, formula (7.21.3), we obtain for all $x$ satisfying $\big|1-\frac{ab}{x+a}\big| < 1$, i.e., for all $x> a\big(\frac{b}{2} - 1 \big)$, the estimate
\begin{align*}
\textstyle \big|R_N(\tilde{g}_{b,x}) \big|
&\le \textstyle \frac{1}{x+a} \big( 1-\frac{ab}{x+a}\big)^{2N}
 \sum\limits_{n=0}^{\infty} \big| 1-\frac{ab}{x+a}\big|^n\, 
\sum\limits_{k=1}^N w_{N,k}^{(L)} \, {\mathrm e}^{t_{N,k}^{(L)}/2} \\
&\le\textstyle \frac{2}{x+a} \big( 1-\frac{ab}{x+a}\big)^{2N}
 \sum\limits_{n=0}^{\infty} \big| 1-\frac{ab}{x+a}\big|^n\\
 &= \textstyle \frac{2}{x+a} \big( 1-\frac{ab}{x+a}\big)^{2N}  \frac{1}{1-\big| 1-\frac{ab}{x+a}\big|}
 =  c\big( 1-\frac{ab}{x+a}\Big)^{2N},
\end{align*}
where we 
used that $\sum\limits_{k=1}^N w_{N,k}^{(L)} \, {\mathrm e}^{t_{N,k}^{(L)}/2} \le2$. This estimate follows with (\ref{GL}) and (\ref{GLerror})
from 
\begin{align*} \textstyle  0 \le \sum\limits_{k=1}^N w_{N,k}^{(L)} \, {\mathrm e}^{t_{N,k}^{(L)}/2} = \int\limits_0^{\infty}
{\mathrm e}^{s/2}\, {\mathrm e}^{-s} \, {\mathrm d}s -\frac{(N!)^2}{(2N)!} \, ({\mathrm e}^{\cdot/2})^{(2N)}(\xi) < \int\limits_0^{\infty}
{\mathrm e}^{s/2}\, {\mathrm e}^{-s} \, {\mathrm d}s
= 2,
\end{align*}
see also \cite{Xiang12}, page 438. In the last step, we find either 
$c=\frac{2}{x+a}\frac{x+a}{ab}= \frac{2}{ab}$ if $\frac{ab}{x+a} \le1$ or $c=\frac{2}{2(x+a)-ab}$ for $\frac{ab}{x+a} > 1$.
\end{proof}

\begin{remark}\label{rem1}
The error estimate $(\ref{err1})$ can be further improved for fixed $N$ by precomputing the exact 
sums $\sum\limits_{k=1}^N w_{N,k}^{(L)}  L_{n+2N} (t_{N,k}^{(L)})$ for $n=0, \ldots, K-1$ and inserting the corresponding terms.
For every fixed $K \in {\mathbb N}_0$ the error representation $(\ref{errorexact})$ yields
\begin{align}\label{imp}
   |R_N(\tilde{g}_{b,x})| 
  &  < \textstyle c \big|1-\frac{ab}{x+a}\big|^{2N+K} 
  + \frac{1}{x+a} \sum\limits_{n=0}^{K-1} \big|1-\frac{ab}{x+a}\big|^{2N+n}
  \Big|  \sum\limits_{k=1}^N w_{N,k}^{(L)} \, L_{n+2N}(t_{N,k}^{(L)})\Big|
  \end{align}
  with $c$ as in Theorem $\ref{theoer}$. Indeed, $\sum\limits_{k=1}^N w_{N,k}^{(L)}  L_{n+2N} (t_{N,k}^{(L)})$ is significantly smaller than $2$ for small $n$.
We use this error estimate for $K=10$ in our numerical experiments for $N=8$ and $N=10$, and it is very tight.
\end{remark}

\begin{remark}
There have been several attempts to improve error estimates for Gaussian quadrature formulas for 
entire functions, which usually relate to the finite interval $[-1,1]$. For the Gauss-Laguerre case we refer to \textnormal{\cite{Lub83, Xiang12}}. The estimates in \textnormal{\cite{Lub83, Xiang12}} are however asymptotic in nature 
(i.e., hold for $N \to \infty$) and cannot be used for our purpose.
\end{remark}

\subsubsection{Exponential sum approximation for $\frac{1}{x+a}$}

The estimate (\ref{err1}) shows that for $b <2$, a geometrically decreasing error of the Gauss-Laguerre quadrature is achieved for all $x \ge 0$; however, $\lim\limits_{x \to \infty} \big( 1 - \frac{ab}{x+a}\big) =1$. 
To achieve the desired geometric error decay for all $x$, we propose to use modified approximations on consecutive intervals.

\begin{theorem}\label{theo2b}
For a given $\rho>1$  and given $a>0$ let $a_j := \frac{\rho+1}{\rho} \big( \frac{\rho+1}{\rho-1}\big)^j$ 
 and ${b}_j:= a \big(\big(\frac{\rho+1}{\rho-1}\big)^j-1\big)$ for $j \in {\mathbb N}_0$ as well as $a_{-1}:= \frac{\rho-1}{\rho}$.
Then  the function $f_1(x)= \frac{1}{a+x}$ can be approximated by an exponential 
sum of arbitrary length $N \in {\mathbb N}$ with 
\begin{align}\label{approx2}
 \textstyle \Big| \frac{1}{x+a} - \sum\limits_{k=1}^N \, \Big(\frac{w_{N,k}^{(L)}}{a a_j} \,  
 {\mathrm e}^{\frac{t_{N,k}^{(L)}(a_j-1)}{a_j}} \Big)\, 
  {\mathrm e}^{-\frac{t_{N,k}^{(L)} x}{a a_j}} \Big|  \le  \frac{2}{aa_{j-1}} \big( 1-\frac{a a_j}{x+a}\big)^{2N}
  &\le 
\textstyle  \frac{2}{aa_{j-1}} \, \rho^{-2N} \end{align}
for $ x \in [{b}_j, \, {b}_{j+1}]$ and $j \ge 0$, where $t_{N,k}^{(L)}$ are the roots of the Laguerre polynomial $L_{N}$ and $w_{N,k}^{(L)}$ the weights in $(\ref{wLaguerre})$.
\end{theorem} 

\begin{proof} 
Let $x \in [{b}_j, \, {b}_{j+1}]
= \big[ a(\frac{\rho}{\rho+1} a_j -1), \,a(\frac{\rho}{\rho+1} a_{j+1} -1) \big]$,  i.e., $x+a \in \frac{a \rho}{\rho+1} [a_{j}, \, a_{j+1}]$
 for $j \in {\mathbb N}_0$.
Then 
\begin{align*} \textstyle - \frac{1}{\rho}  \le  1- \frac{aa_j}{a+x} \le \frac{1}{\rho}, 
\end{align*}
 and Theorem \ref{theoer} implies
 $$ \textstyle |R_N(\tilde{g}_{a_j,x})| \le c  \big( 1-\frac{aa_j}{x+a} \big)^{2N} \le c \, \rho^{-2N}$$
 with $c = \frac{2}{x+a}   \frac{1}{1-\big| 1-\frac{aa_j}{x+a}\big|} \le \frac{2}{x+a}  \frac{1}{1-\frac{1}{\rho}} \le 
% \frac{2\rho(\rho-1)^{j-1}}{a(1+\rho)^{j}} 
 \frac{2(\rho +1)}{a a_{j} (\rho-1)}
 = \frac{2}{aa_{j-1}}$, where for $x \in  [{b}_0, \, {b}_{1}]$ we set  $a_{-1} := \frac{\rho-1}{\rho}$.
Hence, the assertion follows.
\end{proof}

\begin{remark}
Clearly, the error estimate $(\ref{err0})$ can also serve to find an approximation for $\frac{1}{x+a}$ at 
consecutive intervals, which can be derived similarly to
Theorem $\ref{theo2b}$. In this case, the following can be shown:

For a given $\rho>1$ and given $a>0$, the function $f_1(x)= \frac{1}{a+x}$ can be approximated by an exponential 
sum of arbitrary length $N \in {\mathbb N}$ with  
\begin{align}\label{approx1}
 \textstyle \Big| \frac{1}{x+a} - \sum\limits_{k=1}^N \, \Big(\frac{w_{N,k}^{(L)}}{a c_j} \,  
 {\mathrm e}^{\frac{t_{N,k}^{(L)}(c_j-1)}{c_j}} \Big)\, 
  {\mathrm e}^{-\frac{t_{N,k}^{(L)} x}{a c_j}} \Big| &\le 
\textstyle  \frac{\sqrt{\pi (N+\frac{1}{2})}}{a c_j}  \, \rho^{-2N} 
 \end{align}
for $ x \in 
\big[ a \big(c_j-1\big), \, a \big(c_{j+1}-1\big)\big]$,
where   $c_j = \big( \frac{2}{\rho} +1\big)^j$ for $j\in {\mathbb N}_0$.

Note that the obtained exponential sum approximation in $(\ref{approx1})$ has exactly the same structure as 
$(\ref{approx2})$, if one replaces $a_j$ by $c_j$. But these values do not coincide. For fixed $\rho>2$ we always have $c_j < a_j$ and for the length of the $j$-th interval we obtain
$$  \textstyle 
a (c_{j+1}-c_{j} ) = a \big(\frac{2}{\rho} + 1\big)^j \frac{2}{\rho} = \frac{2a}{\rho} \big(\frac{\rho+2}{\rho}\big)^j <  \frac{2a}{\rho-1} \big(\frac{\rho+1}{\rho-1}\big)^j = {b}_{j+1} - {b}_j. $$
\end{remark}

\subsection{Approximation on consecutive intervals and numerical experiments}

Our method to find exponential sum approximations (of the same type on every interval) requires very little 
computational effort and is numerically stable in double-precision arithmetic. 
As we will show next, the required parameters for the exponential sum approximations at different intervals
are closely related and can be obtained by a simple coordinate transformation.
Moreover, we derive the same relative error at every interval.

\begin{theorem}\label{theotrans}
Let  $\rho>1$ and $a>0$ be given and let  
$a_j = \frac{\rho+1}{\rho} \big( \frac{\rho+1}{\rho-1}\big)^j$ and ${b}_j:= a \big(\big(\frac{\rho+1}{\rho-1}\big)^j-1\big)$ for $j \in {\mathbb N}_0$ as well as $a_{-1}:= \frac{\rho-1}{\rho}$ be given as before.
Then  we have for  $ x \in [{b}_j, {b}_{j+1}]$ 
 and 
$\tilde{x}:=\frac{x-{b}_j}{{b}_{j+1}-{b}_j} \in [0,1]$
\begin{align*}%\label{approx3}
\textstyle 
    \big|\frac{1}{x+a} -\frac{1}{aa_j} \sum\limits_{k=1}^N   w_{N,k}^{(L)} \,  
 {\mathrm e}^{-t_{N,k}^{(L)}  \frac{1}{\rho+1} \big(\frac{2 \rho \tilde{x}}{\rho-1}- 1\big)}\big| \le 
 \frac{2}{aa_{j-1}} \,  \big( \frac{\rho (2\tilde{x}  -1)+1 }{\rho(2 \tilde{x}-1)+ \rho^2}\big)^{2N} 
\le  \frac{2}{aa_{j-1}} \, \rho^{-2N}.
 \end{align*}
 In particular, the approximating exponential sum depends (up to the pre-factor $a_j$) not on $j$ but only on $\tilde{x}$.
 Furthermore, the relative error 
 \begin{align} \label{relerr1}
 \textstyle  \frac{\frac{1}{x+a} - \sum\limits_{k=1}^N \frac{w_{N,k}^{(L)}}{a a_j} \,  
 {\mathrm e}^{-t_{N,k}^{(L)} \big(\frac{x+a}{aa_j} -1\big)}}{\frac{1}{a+x}}
 &= \textstyle  1 -  \frac{\rho}{\rho+1} \big( 1 + \frac{2 \tilde{x}}{\rho-1}\big) \sum\limits_{k=1}^N   w_{N,k}^{(L)} \,  
 {\mathrm e}^{-t_{N,k}^{(L)}  \frac{1}{\rho+1} \big(\frac{2 \rho \tilde{x}}{\rho-1}- 1\big)}
\end{align}
is independent of $j$. 
 \end{theorem} 
 
\begin{proof}
For  $x={b}_j + \tilde{x} ({b}_{j+1} - {b}_j)$ with $\tilde{x} \in [0,1]$, we obtain the exponential sum approximation from Theorem \ref{theo2b},
\begin{align*}
\textstyle  \sum\limits_{k=1}^N \, \Big(\frac{w_{N,k}^{(L)}}{a a_j} \,  
 {\mathrm e}^{\frac{t_{N,k}^{(L)}(a_j-1)}{a_j}} \Big)\, 
  {\mathrm e}^{-\frac{t_{N,k}^{(L)} x}{a a_j}} 
  &=  \textstyle  \sum\limits_{k=1}^N  \frac{1}{a a_j} \, w_{N,k}^{(L)} \,  
 {\mathrm e}^{-t_{N,k}^{(L)} \big(\frac{x+a}{aa_j} -1\big)} ,
\end{align*}
where by %$a_{j+1}= a_{j} \big(\frac{\rho+1}{\rho-1}\big)$
$b_{j+1}-b_{j}=\frac{2\rho \, a a_{j}}{\rho^{2}-1}$,
\begin{align*}
\textstyle \frac{x+a}{aa_j} -1&= \textstyle\frac{{b}_j + \tilde{x} ({b}_{j+1} - {b}_j) +a}{aa_j} -1 
= \frac{ a(\frac{\rho}{\rho+1} a_j -1)+ 
%\big(a(\frac{\rho}{\rho+1} a_{j+1} -1) - a(\frac{\rho}{\rho+1} a_j -1)\big) 
 \frac{2 \rho \, a a_{j}}{\rho^{2}-1} \tilde{x} 
+a}{aa_j}-1\\
&
= \textstyle \frac{ \frac{\rho}{\rho+1} 
\big(a_j+   \frac{2 a_{j} \tilde{x}}{\rho-1}% ( a_{j+1} -  a_j)
\big)}{a_j}-1
= \textstyle \frac{\rho}{\rho+1} \big( 1 + \frac{2 \tilde{x}}{\rho-1}\big) -1
= \frac{1}{\rho+1} \big(\frac{2 \rho \tilde{x}}{\rho-1}- 1\big)
\end{align*}
and $\frac{x+a}{aa_j} = \frac{\rho}{\rho+1} \big( 1 + \frac{2 \tilde{x}}{\rho-1}\big)$. Thus we obtain
$$ \textstyle \sum\limits_{k=1}^N  \frac{1}{a a_j} \, w_{N,k}^{(L)} \,  
 {\mathrm e}^{-t_{N,k}^{(L)} \big(\frac{x+a}{aa_j} -1\big)} 
 =  \frac{1}{a a_j}\sum\limits_{k=1}^N    w_{N,k}^{(L)} \,  
 {\mathrm e}^{-t_{N,k}^{(L)}  \frac{1}{\rho+1} \big(\frac{2 \rho \tilde{x}}{\rho-1}- 1\big)} $$
 and 
 \begin{align*}
\textstyle  \big( 1- \frac{aa_j}{x+a} \big) &=  
\textstyle 1- \frac{\rho+1}{\rho} \frac{1}{\big( 1 + \frac{2 \tilde{x}}{\rho-1}\big)}
=  1- \frac{\rho+1}{\rho} \frac{\rho-1}{\rho-1 +2 \tilde{x}}
= 1- \frac{\rho+1}{\rho} \big( 1- \frac{2\tilde{x}}{\rho-1 +2 \tilde{x}}\big)\\
& \textstyle = -\frac{1}{\rho} + \frac{ 2 (\rho +1) \tilde{x}}{\rho(\rho-1 +2 \tilde{x})}
=\frac{-(\rho-1 +2 \tilde{x}) + 2 (\rho +1) \tilde{x} }{\rho(\rho-1 +2 \tilde{x})}
=\frac{\rho (2\tilde{x}  -1)+1 }{\rho(2 \tilde{x}-1)+ \rho^2}.
\end{align*}
Moreover, we conclude 
that the relative error 
\begin{align*}
\textstyle  \frac{\frac{1}{x+a} - \sum\limits_{k=1}^N \frac{w_{N,k}^{(L)}}{a a_j} \,  
 {\mathrm e}^{-t_{N,k}^{(L)} \big(\frac{x+a}{aa_j} -1\big)}}{\frac{1}{a+x}}
 &= \textstyle 1 -   \frac{x+a}{a a_j}\Big(\sum\limits_{k=1}^N w_{N,k}^{(L)}
 {\mathrm e}^{-t_{N,k}^{(L)} \big(\frac{x+a}{aa_j} -1\big)} \Big) \\
 &= \textstyle  1 -  \frac{\rho}{\rho+1} \big( 1 + \frac{2 \tilde{x}}{\rho-1}\big) \sum\limits_{k=1}^N   w_{N,k}^{(L)} \,  
 {\mathrm e}^{-t_{N,k}^{(L)}  \frac{1}{\rho+1} \big(\frac{2 \rho \tilde{x}}{\rho-1}- 1\big)}
\end{align*}
is indeed independent of $j$, see also Figures $\ref{fig1}$ and $\ref{fig2}$.
Note that the same is true for the relative error bounds that follow from  Theorem $\ref{theo2b}$, i.e., we find that
$$ \textstyle   \frac{2}{aa_{j-1}} \Big( 1-\frac{aa_j}{x+a}\Big)^{2N}\, (x+a) =  
\frac{2(x+a)(\rho+1)}{aa_{j} (\rho-1)} \Big( \frac{\rho (2\tilde{x}  -1)+1 }{\rho(2 \tilde{x}-1)+ \rho^2}\Big)^{2N}
= \frac{2\rho}{\rho-1} \big( 1 + \frac{2 \tilde{x}}{\rho-1}\big)
\Big( \frac{\rho (2\tilde{x}  -1)+1 }{\rho(2 \tilde{x}-1)+ \rho^2}\Big)^{2N}
$$
is independent of $j$.
\end{proof}

Theorem \ref{theotrans} shows that one does not need to compute new weights and 
exponent parameters for the approximating exponential sums at different intervals 
$[{b}_j, {b}_{j+1}]$, since a simple transformation to the coordinate $\tilde{x}$ can be applied instead.
\medskip

\begin{remark}
We note that other known approximations of $\frac{1}{x+a}$ on an interval $[0, b]$, $b =b_{1}>0$, which can for example be found in \textnormal{\cite{Braess86,Braess05}}
or computed by the Remez algorithm in \textnormal{\cite{Hackbusch19}},
may also be used to obtain equally good approximations on further intervals $[{b}_j, {b}_{j+1}]$, 
following the same approach as in Theorem $\ref{theo2b}$. Also the \textnormal{ESPIRA} algorithm \textnormal{\cite{Derevianko23}} yields a very good approximation error for $f_{1}$ on small intervals, thereby outperforming %the approximations based on Prony's method \textnormal{\cite{Plonka14}} as
the matrix pencil method and \textnormal{ESPRIT}, see \textnormal{\cite{Derevianko23}}, Example $6.1$. Numerically, the Remez algorithm as well as \textnormal{ESPIRA} provide smaller approximation errors than those obtained by
our explicit approximation formula in $(\ref{approx2})$ on the interval $[0, b_{1}]$.
 Note, however, that these computational methods require some
non-linear optimisation and do not yield corresponding error estimates.
For the function $f(x)= \frac{1}{x+a}$ we can find in \textnormal{\cite{Braess26}}, Theorem \textnormal{4.3}, an estimate for the best approximation by an exponential sum of length $N$ on the interval $[0, b-a]$ of the form $E_N(f, [0,b-a]) \le 2f(0) \sigma_{2N}= \frac{2}{a}\sigma_{2N}$, where $\sigma_{2N}$ can be bounded by $4 \exp\Big(\frac{-N\pi^2}{\ln(4(b-a))}\Big)$.

Assume that we know an approximation 
$$ \textstyle  \big| \frac{1}{x+a} - \sum\limits_{k=1}^N c_k \, {\mathrm e}^{-T_k \, x} \big| < \epsilon, \qquad x \in [0, {b}]. $$
Let now $a_j:=\big( \frac{{b}+a}{a} \big)^j$ and ${b}_j := a (a_j-1)$. Then it follows that 
\begin{align*}
 \textstyle  \big| \frac{1}{x+a} - \frac{1}{a_j} \sum\limits_{k=1}^N c_k \, {\mathrm e}^{-T_k \, \big( \frac{x}{a_j} - a(1 - \frac{1}{a_j})\big)} \big| < \frac{\epsilon}{a_j} \qquad x \in [{b}_j, {b}_{j+1}].
 \end{align*}
Indeed, letting $\tilde{x}_j:= \frac{x-{b}_j}{{b}_{j+1}-{b}_j} =  \frac{x-a(a_j-1)}{{b} \,  a_j} \in [0, 1]$
for $x \in [{b}_j, {b}_{j+1}] $,
we observe that 
 \begin{align*}
\textstyle  \frac{1}{x+a}  = \frac{1}{a+{b}_j+ (x-{b}_j)} = 
\frac{1}{a {a}_j}+ \tilde{x}_j {b} a_j
 =\frac{1}{a_j} \frac{1}{a + \tilde{x}_j {b}}
 \end{align*}
 with $\tilde{x}_j {b} \in [0, {b}]$ such that 
 \begin{align*}
\textstyle \Big| \frac{1}{a_j} \Big(\frac{1}{a + \tilde{x}_j {b}}-  
 \sum\limits_{k=1}^N c_k \, {\mathrm e}^{-T_k \, {b} \tilde{x}_j} \Big)
\Big| < \frac{\epsilon}{a_j},
\end{align*}
where ${b} \tilde{x}_j = \frac{x - a(a_j-1)}{a_j}$.
 \end{remark}

\paragraph{Numerical experiments.}
The proposed approach for approximating $\frac{1}{a+x}$ by an exponential sum enables us to 
achieve very small approximation errors with very short exponential sums. 
Once the nodes and weights of the Gauss-Laguerre quadrature formula are known, these results can be computed in double-precision arithmetic.
We give two numerical examples for $N=8$ and for $N=10$, and provide the corresponding nodes and weights with at least $32$ digits of precision in 
Tables \ref{tab1} and \ref{tab2}.
Note that for a given $x \ge 0$ the index $j$ such that ${b}_j \le x < {b}_{j+1}$ can be easily found using a while-loop.
A {MATLAB} implementation of the algorithm can be found at 
\texttt{http://na.math.uni-goettingen.de/en/Software/index.html}.

\begin{example}
We approximate $\frac{1}{x+a}$  with $a=1$ by an exponential sum of length $N=8$.
 Applying Theorem $\ref{theo2b}$ for $\rho= 4$, we choose
$a_j= \frac{5}{4} \big(\frac{5}{3}\big)^j$, $j \ge -1$, and obtain 
$$  \textstyle \Big| \frac{1}{x+1} - \sum\limits_{k=1}^8 \, \Big(\frac{3^j \cdot 4 \cdot w_{8,k}^{(L)}}{5^{j+1} } \,  
 {\mathrm e}^{t_{8,k}^{(L)}(1 -  \frac{4 \cdot 3^{j}}{5^{j+1}})} \Big)\, 
  {\mathrm e}^{-\frac{4 \cdot 3^{j} \cdot  t_{8,k}^{(L)} \cdot x}{5^{j+1} }} \Big| \le 
\textstyle \frac{8 \cdot 3^{j-1}}{5^{j}} \big(1- \frac{5^{j+1}}{4 \cdot 3^j (x+1)}\big)^{16} \le \frac{8 \cdot 3^{j-1}}{5^{j}}\, 16^{-8} $$
for all 
$x \in  \big[(\frac{5}{3})^j -1, \, (\frac{5}{3})^{j+1} -1 \big]$.
We have to consider the intervals $[0, \frac{2}{3}]$, $[\frac{2}{3}, \frac{16}{9}]$, 
$[\frac{16}{9}, \frac{98}{27}]$, etc., where the length of the intervals grows by the factor $\frac{5}{3}$. 
The computed absolute error and relative error
are illustrated in Figure $\ref{fig2}$ (left column). For $N=8$ and $j=0$, 
$(\ref{approx2})$ provides an error estimate with a bound smaller than $\frac{8}{3} 16^{-8} \approx 6.2088 \cdot 10^{-10}$.
 The improved error bound according to $(\ref{imp})$ with $K=10$, i.e., with the first $10$ terms $\sum\limits_{k=1}^8 w_{8,k}^{(L)}  L_{n+16} (t_{8,k}^{(L)})$, $n=0, \ldots, 9$, being computed exactly, yields an error smaller than $10^{-13}$, which is nearly tight, see Figure $\ref{fig2}$, {middle} left.

 \begin{figure}[bhtp]
\begin{center}
	\includegraphics[scale=0.35]{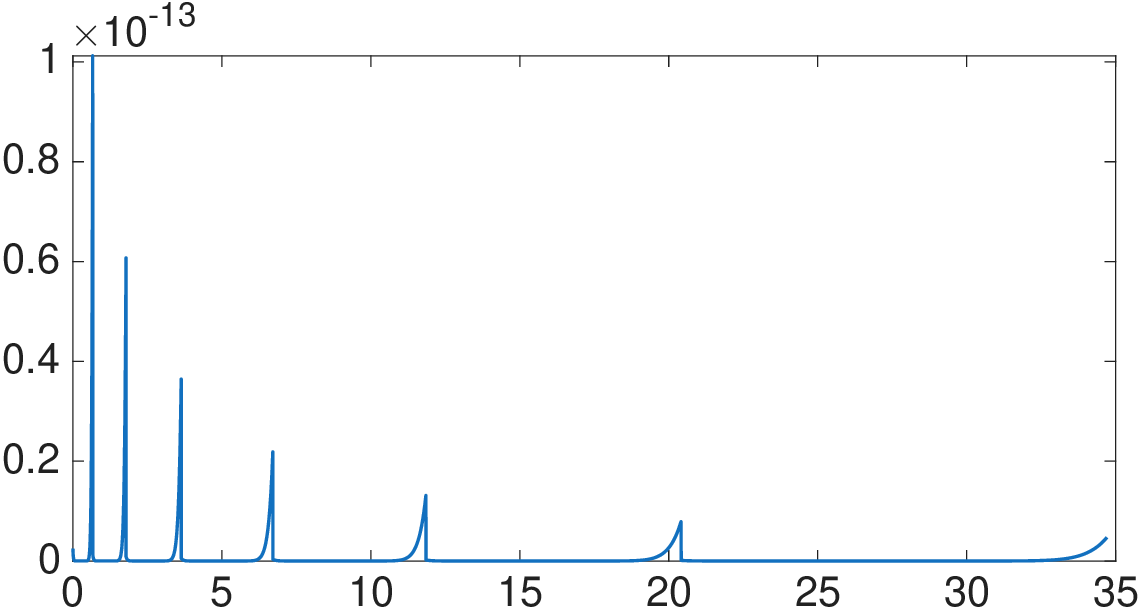} \hspace{5mm}
	\includegraphics[scale=0.35]{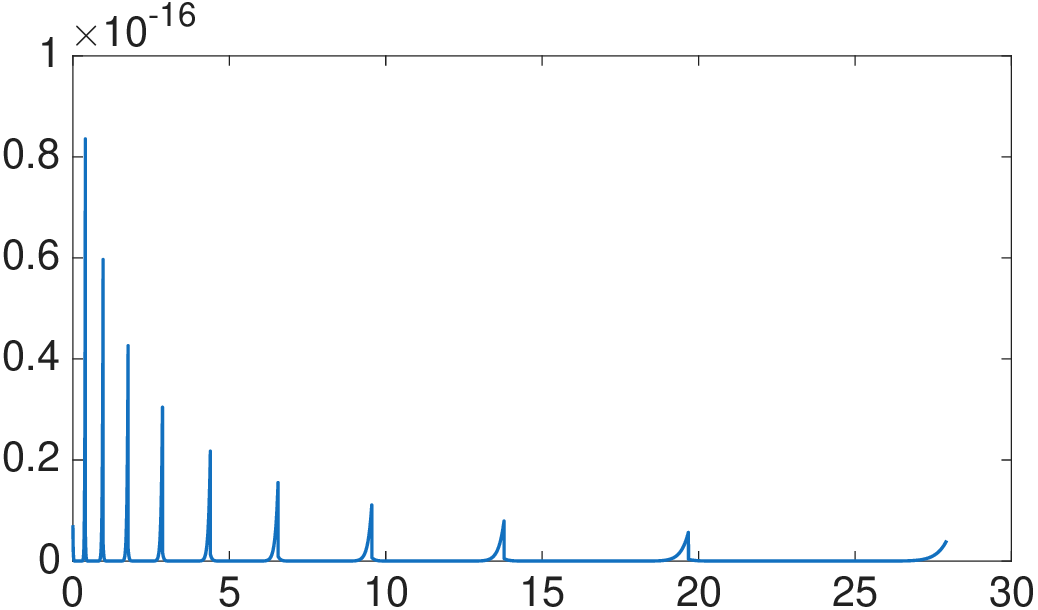} \\
	\includegraphics[scale=0.35]{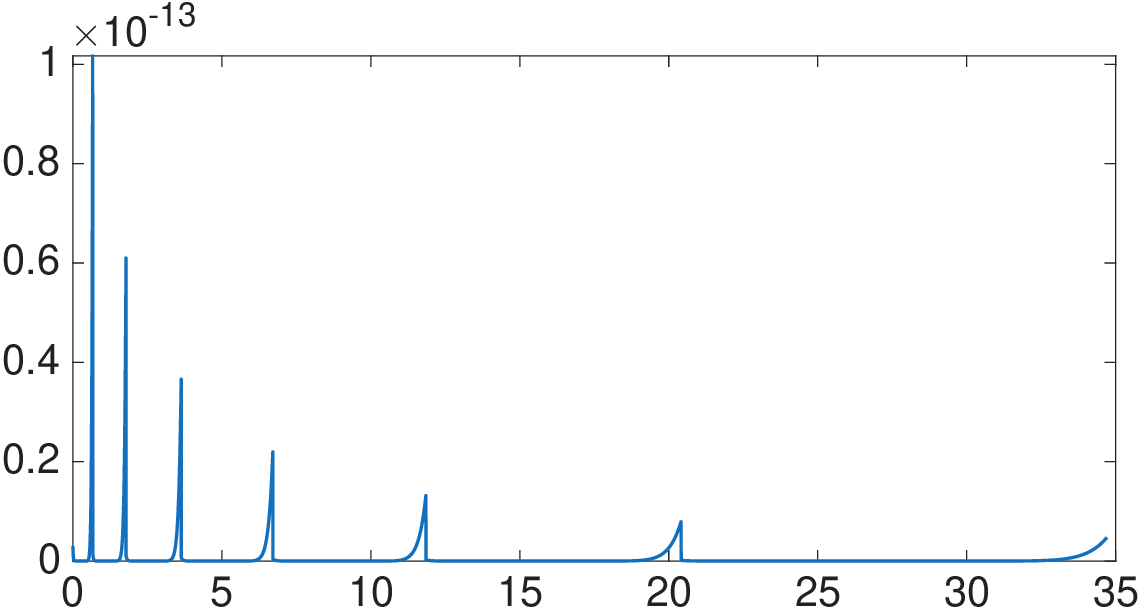}\hspace{5mm}
	\includegraphics[scale=0.35]{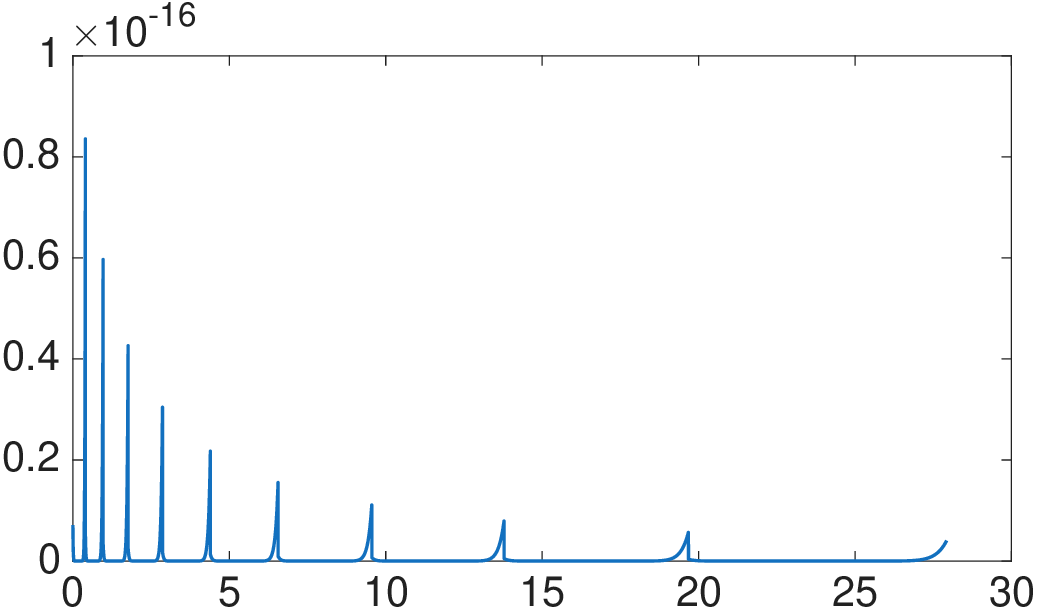}\\
	\includegraphics[scale=0.35]{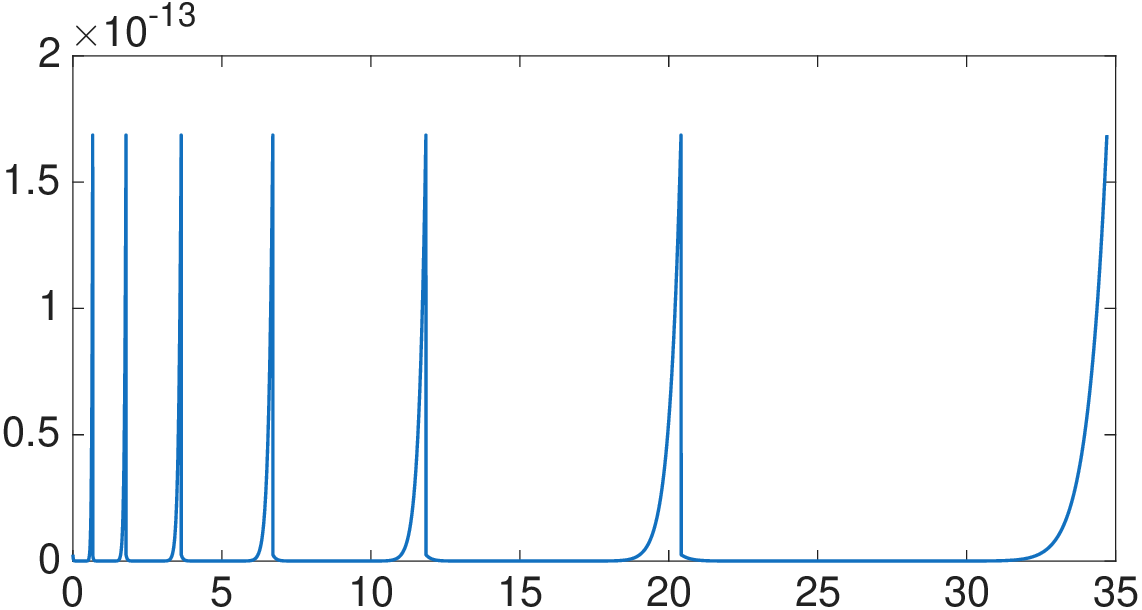}  \hspace{5mm}
         \includegraphics[scale=0.35]{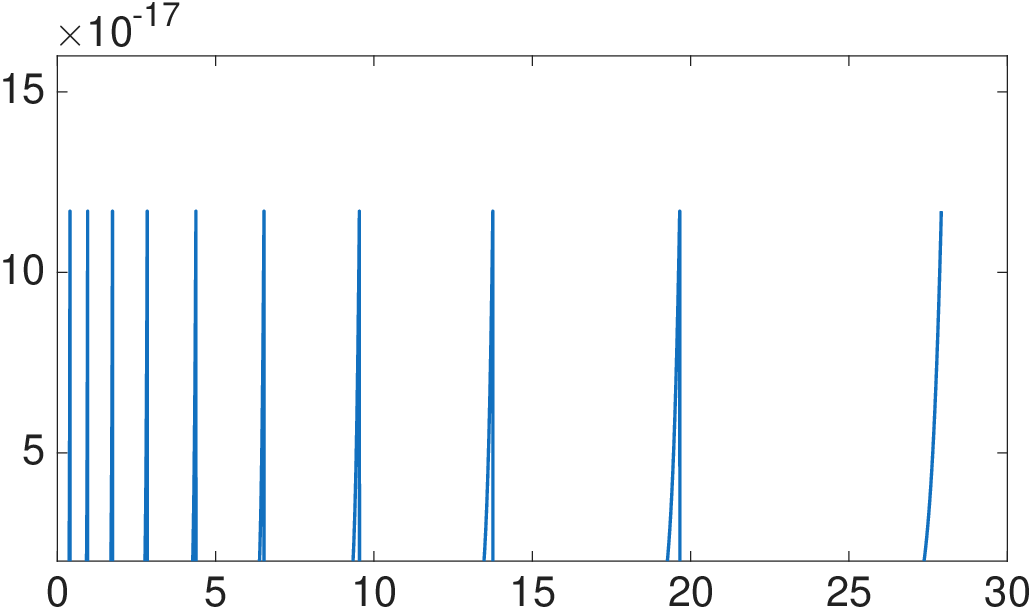} 
	\end{center}
\captionsetup{aboveskip=0pt, justification=justified, labelfont=small, labelsep=space, labelfont=bf}
\caption{\small
Illustration of the error of $ \big|\frac{1}{x+1} - 
 {\sum\limits_{k=1}^N  \Big( \frac{w_{N,k}^{(L)}}{a_{j}} 
 {\mathrm e}^{t_{N,k}^{(L)}(1 -  \frac{1}{a_{j}})} \Big) 
  {\mathrm e}^{-\frac{t_{N,k}^{(L)} x}{a_{j}}}\big|}$
for $N=8$,  $\rho=4$,  $a_j= \frac{5}{4} \big(\frac{5}{3}\big)^j$, and $x \in [0,34.7225)$  (7 intervals) (left column), and for $N=8$,  $\rho=6$,  $a_j= \frac{7}{6}  \big(\frac{7}{5}\big)^j$, and $x \in [0,27.9255)$ (10 intervals) (right column).
Top: absolute numerical errors for $\rho=4$ (left) and $\rho = 6$ (right). 
Middle: theoretical error bounds according to (\ref{imp}) with $K=10$ for $\rho=4$ (left) and $\rho=6$ (right).
Bottom: relative errors for $\rho=4$ (left) and $\rho = 6$ (right). 
}
\label{fig2}
\end{figure}

 \begin{table}[htbp]
\scriptsize
\caption{\small Nodes and weights for Gauss-Laguerre quadrature for $N=8$.}
\begin{center}
\begin{tabular}{|c|l|l|}
\hline
$k$ & nodes $t_{8,k}^{(L)}$ &   weights $w_{8,k}^{(L)}$ \\
\hline
1 & 0.17027963230510099978886185660830 & 0.36918858934163752992058283937570394411528484\\
2 & 0.90370177679937991218602022355509 &  0.41878678081434295607697858133333431953107646\\
3 & 2.25108662986613068930711836696864 & 0.17579498663717180569965986677677365777322531\\
4 & 4.26670017028765879364942182690064 &  0.03334349226121565152213253493440643668469857\\
5 & 7.04590540239346569727932548211936 & 0.00279453623522567252493892414792863974992724\\
6 & 10.7585160101809952240599567880320 & 0.00009076508773358213104238501493357342695903\\
7 & 15.7406786412780045780287611584028 & 0.00000084857467162725315448680183089320313175\\
8 & 22.8631317368892641057005342974131 & 0.00000000104800117487151038161508853551569680\\

\hline
\end{tabular}
\end{center}
\label{tab1}
\end{table}

\medskip
 
 Taking $N=8$ and $\rho=6$, we similarly get $a_j= \frac{7}{6}  \big(\frac{7}{5}\big)^j$, 
 intervals $\big[(\frac{7}{5})^j -1, \,(\frac{7}{5})^{j+1} -1 \big]$, and the error bound $(\ref{imp})$ with $K=10$ 
 yields an error less than 
 $8 \cdot 10^{-17}$, which is below the double-precision limit. The computed absolute error and relative error 
 are illustrated in Figure $\ref{fig2}$ (right column). 
 To reach an error below the double-precision limit, we had to compute the function to be approximated in higher-precision arithmetic.
\end{example}

\begin{example}
Next, we approximate $\frac{1}{x+a}$  with $a=1$ using an exponential sum of length $N=10$.
 According to Theorem $\ref{theo2b}$
 we obtain the theoretical error estimate 
\begin{align*} \textstyle    \Big|  \frac{1}{x+1} -  \sum\limits_{k=1}^{10} \, 
\frac{w_{10,k}^{(L)}}{a_{j}} \,  
 {\mathrm e}^{-t_{10,k}^{(L)}\big(\frac{(x+1)}{a_{j}}-1\big)} \Big|
\le 
\textstyle \frac{2}{a_{j-1}} 
\Big( 1- \frac{a_{j}}{x+1}\Big)^{20}
\le  %\frac{2\rho(\rho-1)^{j-1}}{a(\rho+1)^{j}} 
\frac{2}{a_{j-1}}  \, \rho^{-20}
\end{align*}
for all 
$x \in  \big[ \big(\big(\frac{\rho+1}{\rho-1}\big)^j -1\big), \,  \big(\big(\frac{\rho+1}{\rho-1}\big)^{j+1} -1) \big]$ with $a_{j}= \frac{\rho+1}{\rho} \big( \frac{\rho+1}{\rho-1}\big)^{j}$.
For $\rho=3$, we have $a_{j}= \frac{2^{j+2}}{3}$ and consider the intervals $[0, 1]$, $[1, \, 3]$, 
$[3, 7]$ etc., where the length of the intervals grows by the factor $2$. 
The computed absolute error and relative error are illustrated in Figure $\ref{fig1}$ (left column).
The error estimate above gives $\frac{3}{2^{j}} 9^{-10} < \frac{8.61}{2^{j}} \cdot 10^{-10}$.
Employing the improved error bound $(\ref{imp})$, which is obtained by computing the sums 
$\sum\limits_{k=1}^{10} w_{10,k}^{(L)}  L_{n+20} (t_{10,k}^{(L)})$ for $n=0, \ldots , 9$ exactly, we obtain the error illustrated in Figure $\ref{fig1}$, middle (left).

For $\rho=4$ we have $a_{j}= \frac{5}{4} \big(\frac{5}{3}\big)^{j}$ and smaller intervals $[0,\frac{2}{3}]$, $[\frac{2}{3}, \frac{16}{9}]$ etc.\ must be considered.
The computed absolute error and relative error for $\rho=4$ are illustrated in Figure $\ref{fig1}$ (right), together with the theoretical error estimate $(\ref{imp})$ for $K=10$.

 \begin{table}[htbp]
\scriptsize
\caption{\small Nodes and weights for Gauss-Laguerre quadrature for $N=10$.}
\begin{center}
\begin{tabular}{|c|l|l|}
\hline
$k$ & nodes $t_{10,k}^{(L)}$ &   weights $w_{10,k}^{(L)}$ \\
\hline
1 &  0.13779347054049243083077250565271 &  0.308441115765020141547470834677860695628728886538337442 \\
2 &  0.72945454950317049816037312167608  & 0.401119929155273551515780309912819514795483616962113018  \\
3 &  1.80834290174031604823292007575061 &  0.218068287611809421588648523474646726742778538412188941  \\
4 &  3.40143369785489951448253222140839 & 0.062087456098677747392902129313517953695909065683802092  \\
5 &  5.55249614006380363241755848686876  &  0.009501516975181100553839072194171991225862450401579753  \\
6 &  8.33015274676449670023876719727452  &    0.000753008388587538775455964353675663901792039140143629  \\
7 &  11.8437858379000655649185389191416   & 0.000028259233495995655674225638268500212828033164744375 \\
8 &  16.2792578313781020995326539358336  & 0.000000424931398496268637258657665974712354648108019864  \\
9 &   21.9965858119807619512770901955945 &   0.000000001839564823979630780921535224355938247982612777 \\
10&  29.9206970122738915599087933407992  &  0.000000000000991182721960900855837754728324473606458109
 \\
\hline
\end{tabular}
\end{center}
\label{tab2}
\end{table}

\end{example}

 \begin{figure}[htb]
\begin{center}
	\includegraphics[scale=0.35]{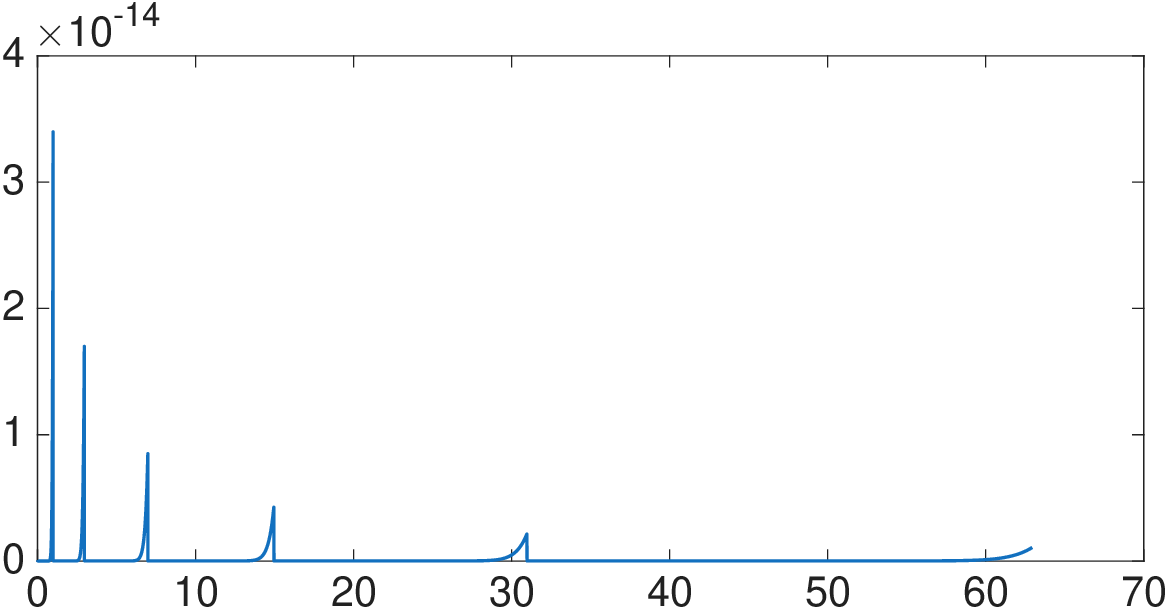} \hspace*{5mm}
	\includegraphics[scale=0.35]{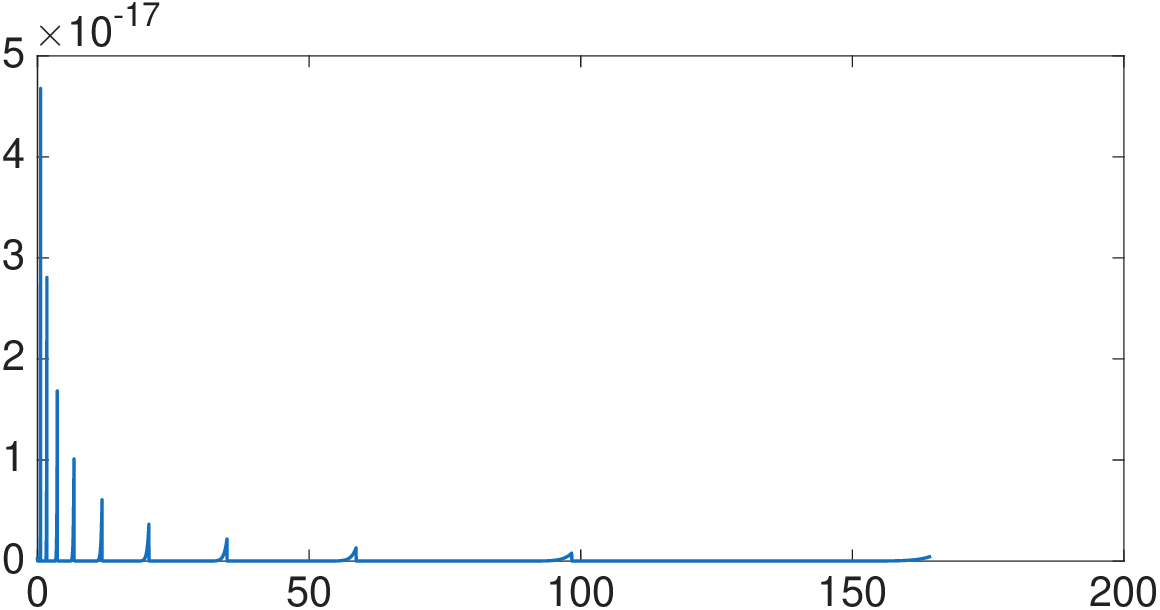} \\
	\includegraphics[scale=0.35]{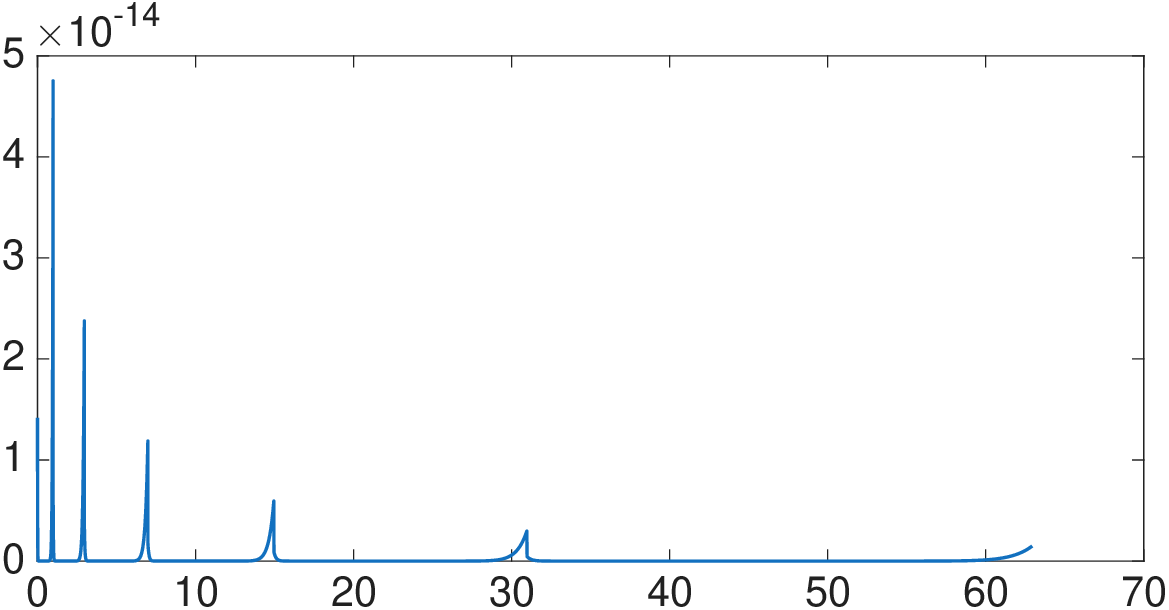} \hspace*{5mm}
	\includegraphics[scale=0.35]{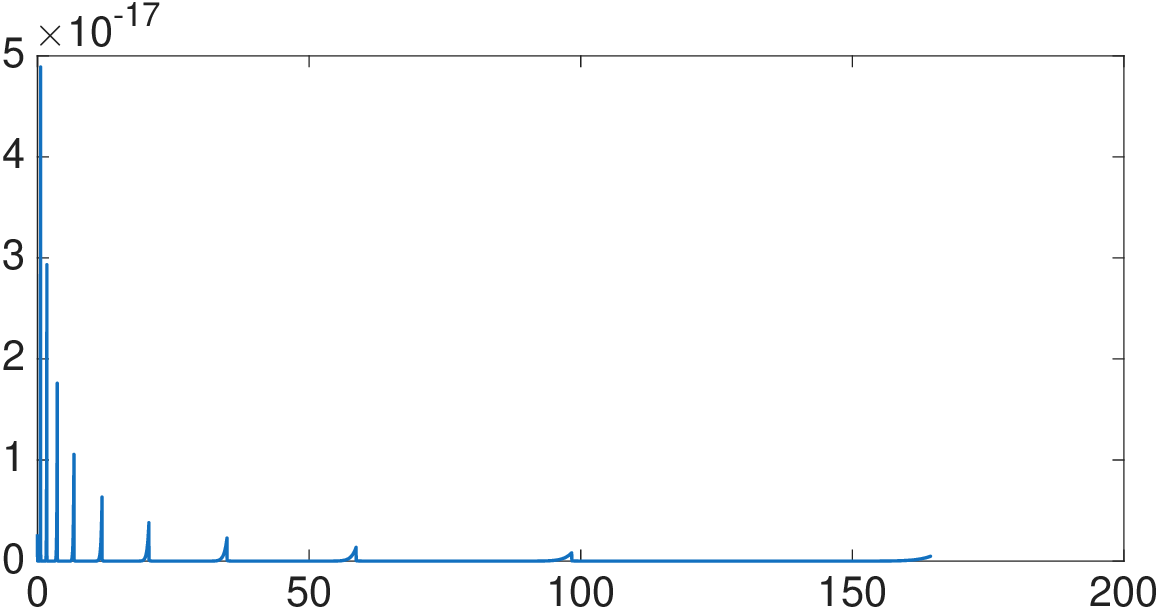}\\
	\hspace*{1mm} \includegraphics[scale=0.35]{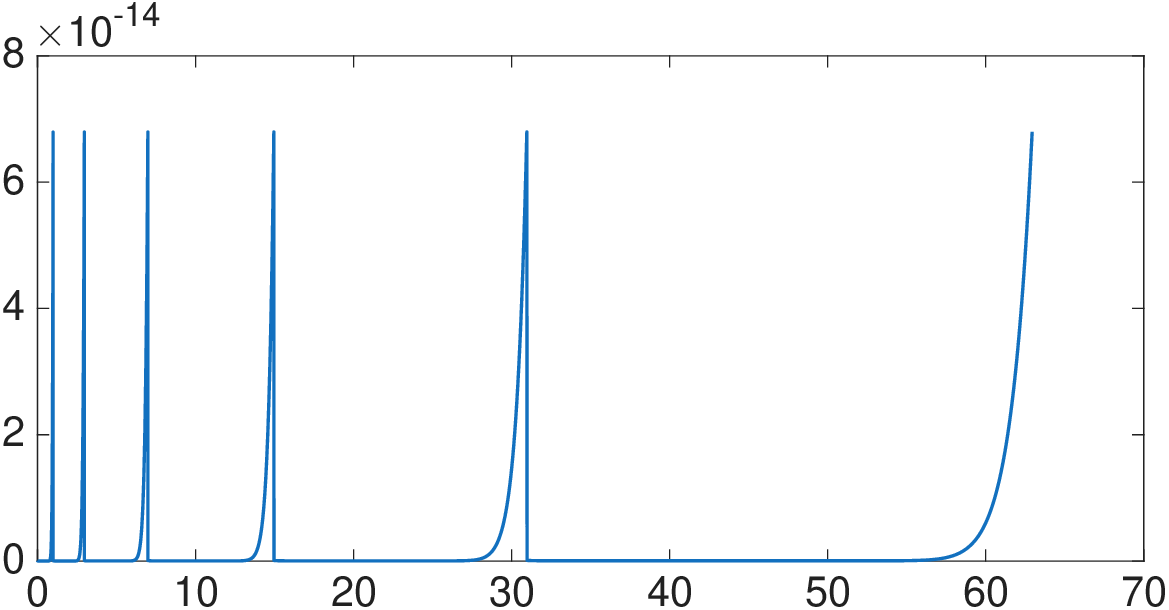} \hspace*{5mm}
         \includegraphics[scale=0.35]{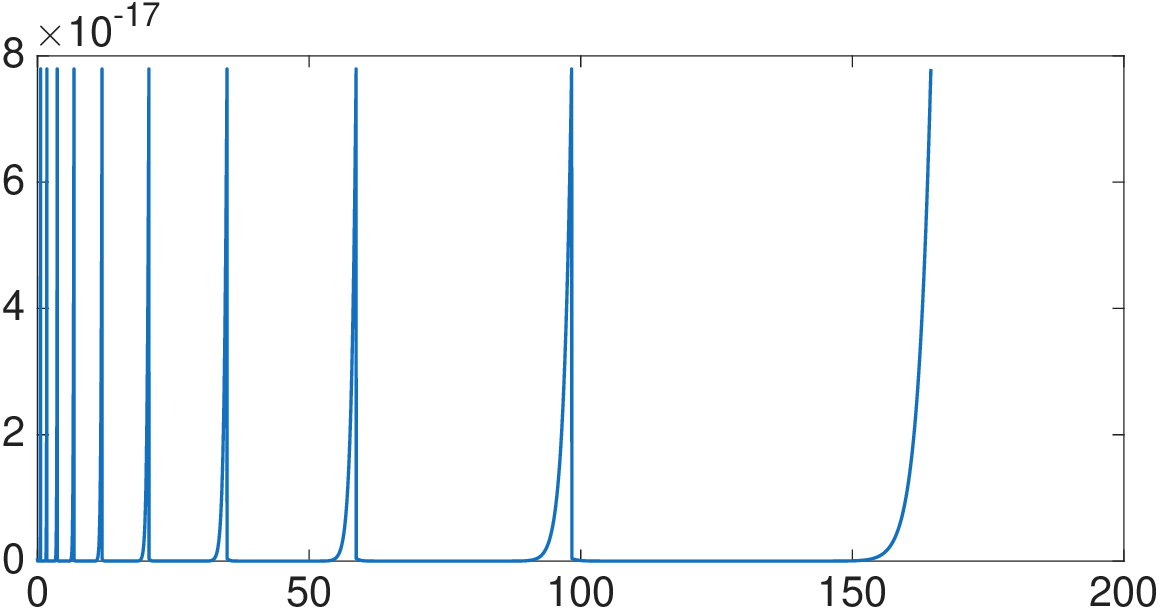} 
	\end{center}
\captionsetup{aboveskip=0pt, justification=justified, labelfont=small, labelsep=space, labelfont=bf}
\caption{\small
Illustration of the error of $    \big|  \frac{1}{x+1} -  \sum\limits_{k=1}^{10} \, 
\frac{w_{10,k}^{(L)}}{a_{j}} \,  
 {\mathrm e}^{-t_{10,k}^{(L)}\big(\frac{(x+1)}{a_{j}}-1\big)} \big|$
%
%\big|\frac{1}{x+1} - 
%\sum\limits_{k=1}^N \frac{(\rho-1)^j \rho  w_{N,k}^{(L)}}{(\rho+1)^{j+1}}  
% {\mathrm e}^{-t_{N,k}^{(L)}(\frac{\rho (\rho-1)^j(x+1)}{(\rho+1)^{j+1}} -1)} \big|$
for $N=10$,  $\rho=3$,  $a_{j}= \frac{2^{j+2}}{3}$, and $x \in [0,63)$ (6 intervals) (left column), and for $N=10$,  $\rho=4$, $a_{j}= \frac{5}{4} \big(\frac{5}{3}\big)^{j}$ and 
$x \in [0,164.3817)$ (10 intervals) (right column).
Top: absolute numerical errors for $\rho=3$ (left) and $\rho = 4$ (right). 
Middle: error  bounds according to (\ref{imp}) for $\rho=3$ (left) and $\rho=4$ (right).
Bottom: relative errors for $\rho=3$ (left) and $\rho = 4$ (right). 
}
\label{fig1}
\end{figure}

\subsection{Application to approximate the logarithm function with high precision}
\label{ApproxLog}

We apply the obtained representation of $\frac{1}{x+a}$ to derive
a short exponential sum representation of the logarithm{$\log(x+1)$ for $x \geq 0$.\\
We employ the same notation as in the previous subsections. Assume that $\rho >1$, $a=1$,  $\gamma:=\frac{\rho+1}{\rho-1}$, ${b}_j = a (\gamma^j - 1) = \gamma^j - 1$ for $j \in {\mathbb N}_0$ and $a_j= \frac{\rho+1}{\rho}  \gamma^j$ for $j \in {\mathbb Z}$, $j \ge -1$, as in Theorem \ref{theo2b}.
Thus, for $x \in [{b}_j, {b}_{j+1})$, $j \ge 0$, we find 
\begin{align*}
	\nonumber
	\log(x+1) &=  \textstyle  \int\limits_0^{{b}_j} \frac{1}{t+1} \, {\mathrm d}t + \int\limits_{{b}_j}^{x} \frac{1}{t+1} \, {\mathrm d}t =  \log(\gamma^{j}) + \int\limits_{{b}_j}^{x} \frac{1}{t+1} \, {\mathrm d}t \\
	\nonumber
	& \textstyle \approx  j \log(\gamma)
	+  	\int\limits_{{b}_j}^{x}\sum\limits_{k=1}^N \, \frac{w_{N,k}^{(L)}}{a_j} \,  
	{\mathrm e}^{t_{N,k}^{(L)}(1 - \frac{1}{a_j})} \, 
	{\mathrm e}^{-\frac{t_{N,k}^{(L)} t}{a_j}} \, {\mathrm d}t\\
	&=  \textstyle j \log(\gamma)
	+  \sum\limits_{k=1}^N \, \frac{w_{N,k}^{(L)}}{a_j} \,  
	{\mathrm e}^{t_{N,k}^{(L)}(1 - \frac{1}{a_j})} \, 
	\int\limits_{{b}_j}^{x}{\mathrm e}^{-\frac{t_{N,k}^{(L)} t}{a_j}} \, {\mathrm d}t\\
	&=  \textstyle j \log(\gamma)
	+  \sum\limits_{k=1}^N \, \frac{w_{N,k}^{(L)}}{t_{N,k}^{(L)}} \,  
	{\mathrm e}^{t_{N,k}^{(L)}(1 - \frac{1}{a_j})} \, 
	\Big({\mathrm e}^{-\frac{t_{N,k}^{(L)} b_j}{a_j}}- {\mathrm e}^{-\frac{t_{N,k}^{(L)} x}{a_j}}\Big)\\
	& = \textstyle j \log(\gamma) +
	\sum\limits_{k=1}^N \frac{w_{N,k}^{(L)}}{t_{N,k}^{(L)}} \,  \Big({\mathrm e}^{\frac{t_{N,k}^{(L)}}{\rho+1}} 
	- {\mathrm e}^{t_{N,k}^{(L)} (1 - \frac{1}{a_j}-\frac{x}{a_j})} 
	\Big).
	%	\label{logx}
\end{align*}
Therefore, we obtain 
\begin{Corollary}\label{corolog}
	%Assume that we have approximated $\frac{1}{x+1}$ by an exponential sum of length 
	%$N\in\N$ with $\rho>1$ as given in Theorem $\ref{theo2b}$. 
	Let $t_{N,k}^{(L)}$, $k=1, \ldots, N$, be the roots of the Laguerre polynomial $L_{N}$ and $w_{N,k}^{(L)}$ the weights in $(\ref{wLaguerre})$.
	Then  for $x \in [{b}_j, {b}_{j+1})$,
	the  error  
	\begin{align} \label{logerror}  \textstyle 
	E(x) = \log(x+1) - \Big(j\log\Big(\frac{\rho+1}{\rho-1}\Big) +\sum\limits_{k=1}^N \frac{w_{N,k}^{(L)}}{t_{N,k}^{(L)}} \,  \Big({\mathrm e}^{\frac{t_{N,k}^{(L)}}{\rho+1}} 
		- {\mathrm e}^{t_{N,k}^{(L)} (1 - \frac{x+1}{a_j})} 
		\Big)\Big)
	\end{align}
	satisfies 
	\begin{align*}  \label{logerroresti} 
	\textstyle |E(x)|  \le \frac{4\rho}{(\rho-1)^2} \, \rho^{-2N}.
	\end{align*}
\end{Corollary}
\begin{proof}
	Application of Theorem \ref{theotrans} with $a=1$ yields for $x \in [{b}_j, {b}_{j+1})$
	\begin{align*} 
		|E(x)| &=  \textstyle  \Big|   \int\limits_{{b}_j}^{x} 
		\frac{1}{t+1} -
		\sum\limits_{k=1}^N \frac{w_{N,k}^{(L)}}{a_j} \,  
		{\mathrm e}^{\frac{t_{N,k}^{(L)}(a_j-t-1)}{a_j}} 
		{\mathrm d}t  \Big| \leq \int\limits_{{b}_j}^{b_{j+1}}  \frac{2}{a_{j-1}} \, \rho^{-2N} {\mathrm d}t\\
		& = \textstyle \frac{2(b_{j+1}-b_j)}{a_{j-1}} \, \rho^{-2N} = \frac{4}{\rho-1} \left(\frac{\rho+1}{\rho-1}\right)^j \frac{\rho}{\rho+1} \left( \frac{\rho-1}{\rho+1}\right)^{j-1}\, \rho^{-2N}\\
		& = \textstyle \frac{4\rho}{(\rho-1)^2} \, \rho^{-2N}.
	\end{align*}
\end{proof}

The approximation of $\log(x+1)$ in (\ref{logerror}) is in fact an exponential sum of length $N+1$ since it contains a constant term 
$j\log\big(\frac{\rho+1}{\rho-1}\big) + \sum\limits_{k=1}^{N} \frac{w_{N,k}^{(L)}}{t_{N,k}^{(L)}} \, {\mathrm e}^{\frac{t_{N,k}^{(L)}}{\rho+1}}$.

The estimate in Corollary \ref{corolog} is quite conservative. For example, for $N=8$ and $\rho=3$ we obtain $|E(x)| \leq 3\cdot3^{-16} \approx 6.9692 \cdot 10^{-8}$ while the computational error is smaller than $10^{-12}$. Figure \ref{figlog} shows the computational errors for the approximation of $\log(x+1)$ by an exponential sum of length $N=8$ with $\rho=2,3,4$ and of length $N=10$ with $\rho=2,3,4$. 

\begin{figure}[htb]
	\begin{center}
		\includegraphics[scale=0.26]{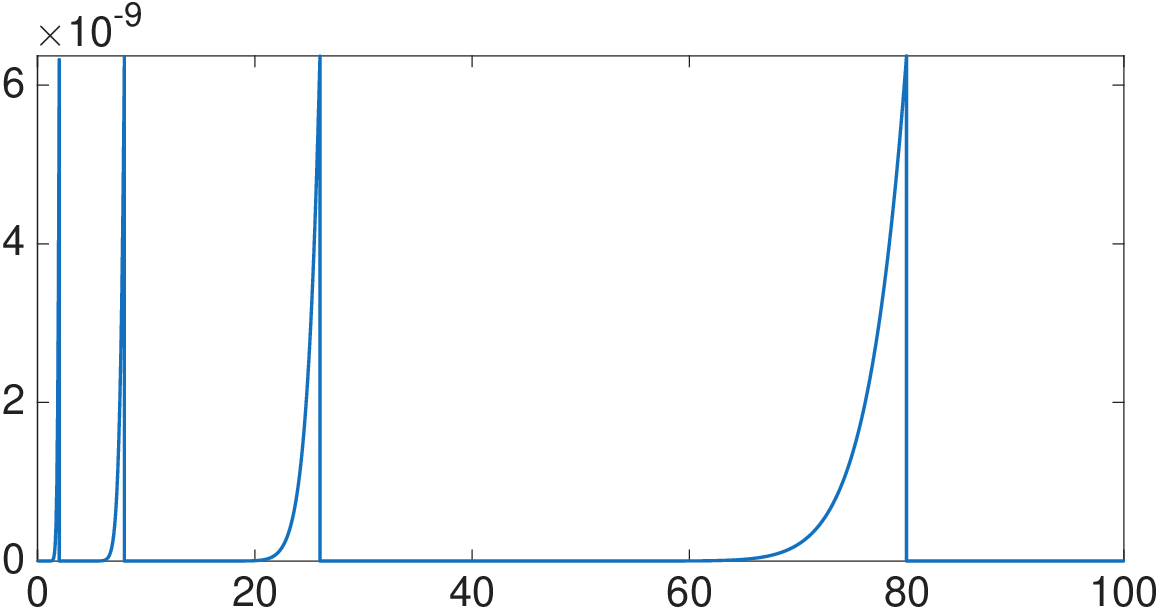} 
		\includegraphics[scale=0.26]{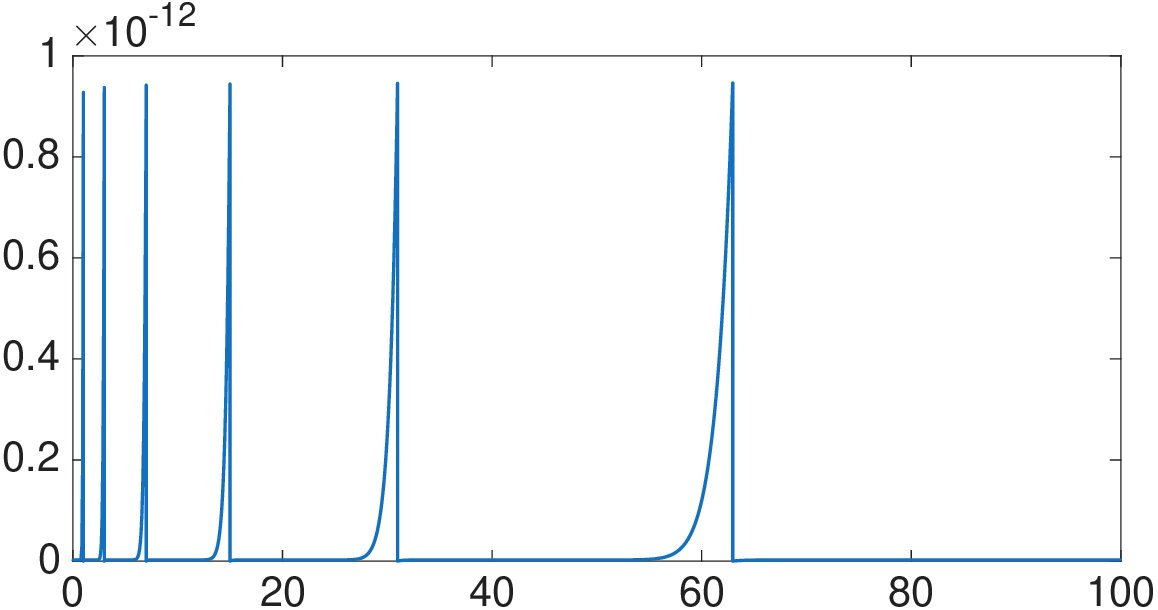}
		\includegraphics[scale=0.26]{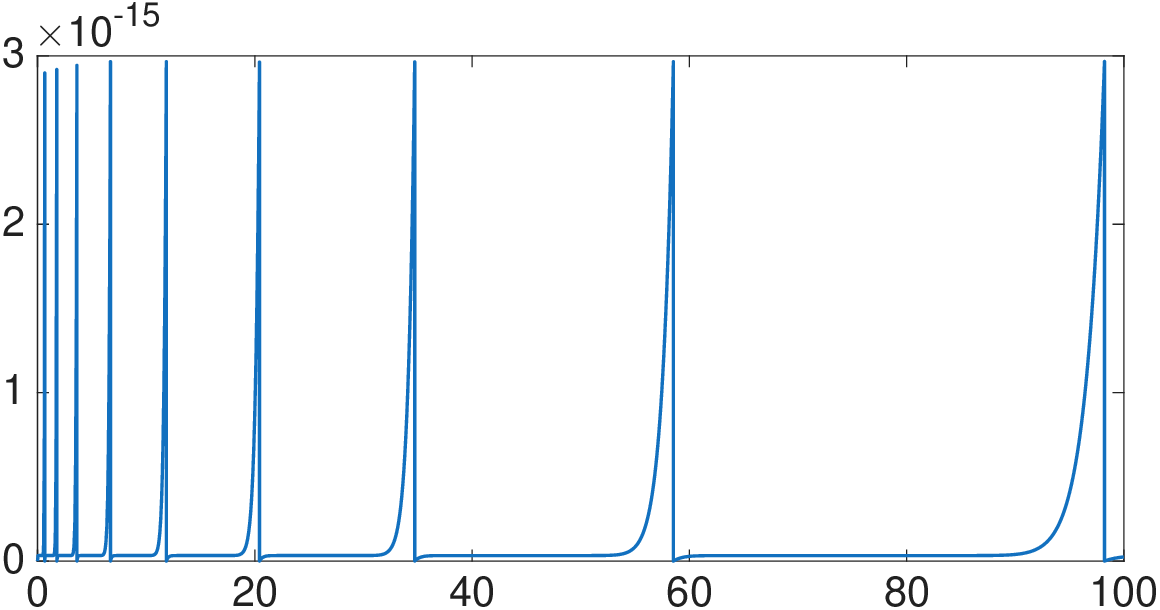} \\
		\includegraphics[scale=0.26]{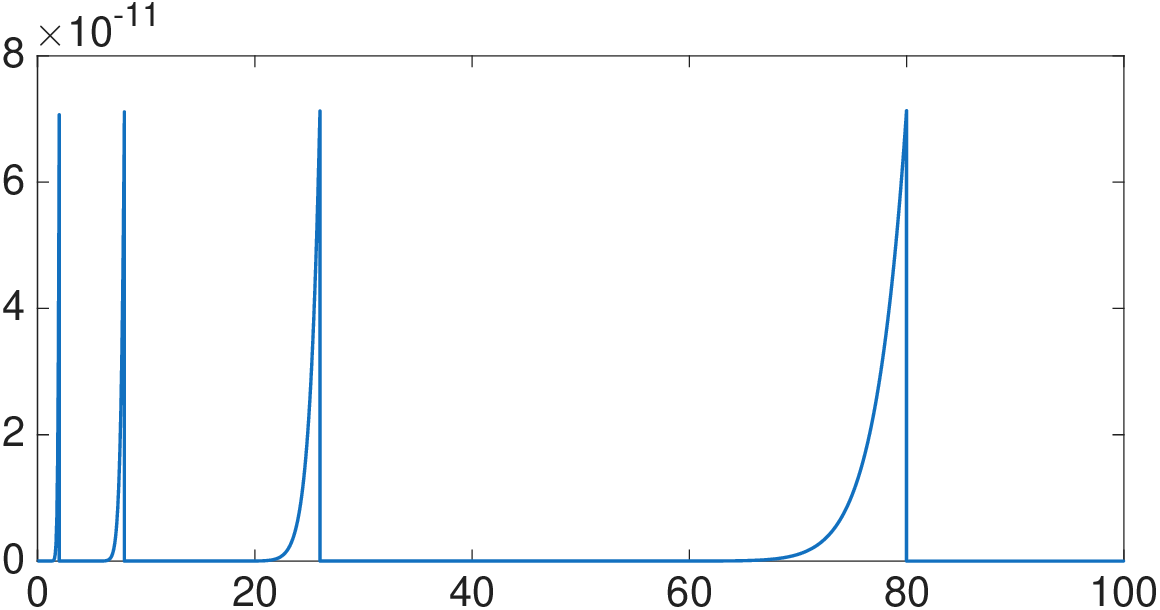} 
		\includegraphics[scale=0.26]{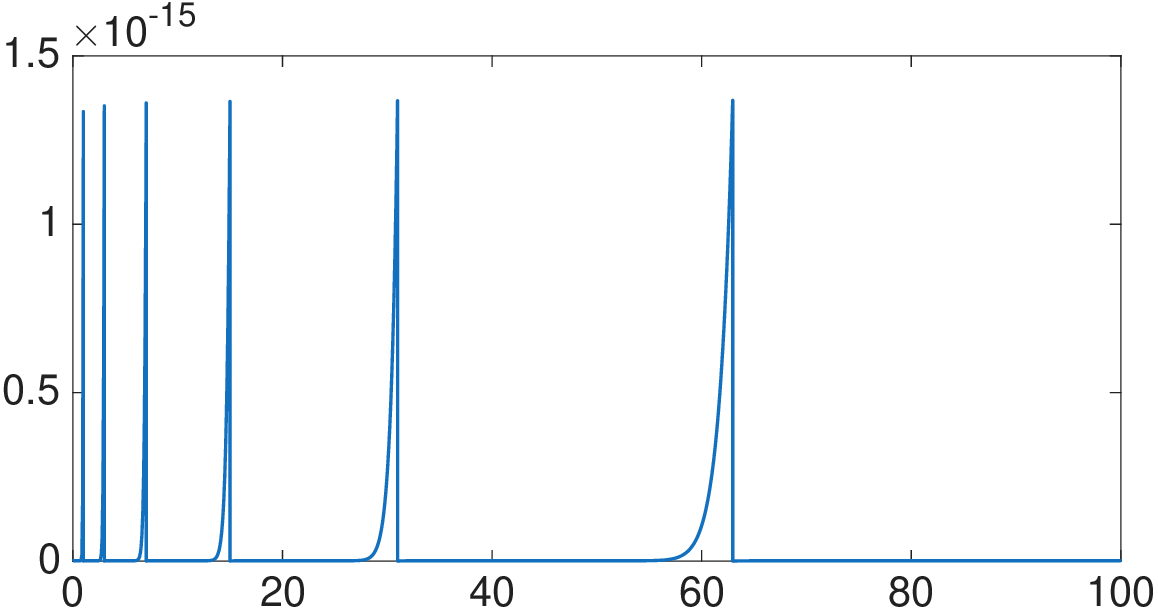}
		\includegraphics[scale=0.26]{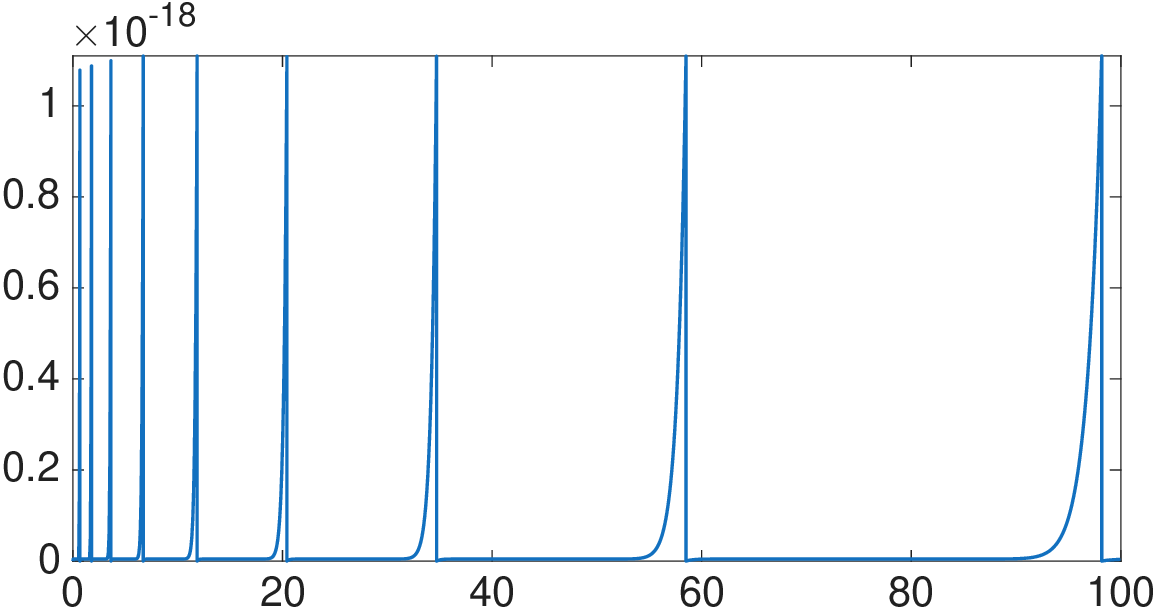} \\
	\end{center}
	\captionsetup{aboveskip=0pt, justification=justified, labelfont=small, labelsep=space, labelfont=bf}
	\caption{\small
		Illustration of the computational error $|E(x)|$ in (\ref{logerror})
		for $x \in [0,100)$. Top: $N=8$ and  $\rho=2$ (left), $\rho=3$ (middle), $\rho=4$ (right).
		Bottom: %Illustration of the computational error $|E(x)|$ in (\ref{logerror})
		 $N=10$ and $\rho=2$ (left), $\rho=3$ (middle), $\rho=4$ (right).
	}
	\label{figlog}
\end{figure}

\begin{corollary}
	The error $E(x)$ in $(\ref{logerror})$ only depends on $\tilde{x}=\frac{x-b_j}{b_{j+1}-b_j}$ via
	\begin{equation*}
	\textstyle 	E(x) = \log\Big(\frac{2}{\rho-1}\tilde{x}+1\Big) - \sum\limits_{k=1}^N   \frac{w_{N,k}^{(L)}}{t_{N,k}^{(L)}} \, 	
		\Big({\mathrm e}^{\frac{t_{N,k}^{(L)}}{\rho+1}} -{\mathrm e}^{-t_{N,k}^{(L)}  \frac{1}{\rho+1} \big(\frac{2 \rho \tilde{x}}{\rho-1}- 1\big)}\Big).
	\end{equation*}
\end{corollary}
\begin{proof}
	We have
	\begin{align*}
		\log(x+1)-\log(b_j+1) &= \textstyle \log\Big(\frac{x+1}{b_j+1}\Big) = \log\Big(\frac{x-b_j+b_j+1}{b_{j+1}-b_j}\frac{b_{j+1}-b_j}{b_j+1}\Big)\\
		& = \textstyle \log\Big(\Big(\tilde{x}+\frac{\rho-1}{2}\Big)\frac{2}{\rho-1}\Big) = \log\Big(\frac{2}{\rho-1}\tilde{x}+1\Big)
	\end{align*}
	and from Theorem \ref{theotrans} we get $1 - \frac{1}{a_j}-\frac{x}{a_j} = 1-\frac{x+1}{a_j} = -\frac{1}{\rho+1} \Big(\frac{2 \rho \tilde{x}}{\rho-1}- 1\Big)$, which proves the assertion.
\end{proof}

\begin{remark}
Our approach to approximate $\log(x+1)$ (or more generally $\log(x+a)$ for $a>0$) by an exponential sum as in Corollary $\ref{corolog}$ can be transferred to approximate the matrix logarithm $\log(x{\mathbf A} + {\mathbf I})$ of a symmetric positive definite matrix ${\mathbf A} \in {\mathbb R}^{n\times n}$ for $x >0$ by a sum of matrix exponentials.
Let ${\mathbf A} = {\mathbf U} {\mathbf D} {\mathbf U}^{T}$ be the singular value decomposition of ${\mathbf A}$, where ${\mathbf U}\in {\mathbb R}^{n \times n}$ is orthogonal and ${\mathbf D} = \diag(d_{1}, \ldots , d_{n})$ a diagonal matrix with positive entries and ${\mathbf I}$ denotes the identity matrix. We observe that 
$\textstyle \log(x{\mathbf A} + {\mathbf I}) = {\mathbf U}
\, \log (x{\mathbf D} + {\mathbf I})  \,  {\mathbf U}^{T}$ since
\begin{align*}
\textstyle \exp({\mathbf U}
\log (x{\mathbf D} + {\mathbf I})  {\mathbf U}^{T}) &=
\textstyle \sum\limits_{k=0}^{\infty} \frac{1}{k!} {\mathbf U} \big(\log (x{\mathbf D} + {\mathbf I}) \big)^k {\mathbf U}^{T}
= {\mathbf U} \Big( \sum\limits_{k=1}^{\infty} \frac{1}{k!} \big(\log
(x{\mathbf D} + {\mathbf I}) \big)^{k}\Big) {\mathbf U}^{T} \\
&= \textstyle  {\mathbf U}
\,  (x{\mathbf D} + {\mathbf I})  \,  {\mathbf U}^{T} = x{\mathbf A} + {\mathbf I},
\end{align*}
and our exponential sum approximation can be applied separately to every component 
$$ \textstyle \log(d_{j}x +1) = \log(d_{j}) + \log(x +\frac{1}{d_{j}}) = \log(d_{j})+
\int\limits_{0}^{x} \frac{1}{t + \frac{1}{d_{j}}} {\mathrm d}t. $$ 
\end{remark}

\section{Approximation of ${\mathrm e}^{-x^2/(2\sigma)}$ by exponential sums}

Next we consider the function $f_2(x)={\mathrm e}^{-x^2/(2\sigma)}$ with $\sigma>0$. This function admits the parametric representation (\ref{rep2}).

\subsection{Approximation based on Gauss-Hermite quadrature}

To derive an exponential sum representation, we apply the Gauss-Hermite quadrature to the integral in (\ref{rep2}).
The Hermite polynomials $H_n$, $n \in {\mathbb N}_0$, are defined by $H_n(s)= (-1)^n \, {\mathrm e}^{s^2} \, \frac{{\mathrm d}^n}{{\mathrm d} s^n} {\mathrm e}^{-s^2}$ and satisfy the orthogonality condition 
$$ \textstyle \int\limits_{-\infty}^{\infty} H_m(s) \, H_n(s) \, {\mathrm e}^{-s^2}\, {\mathrm d}s = 2^n n! \, \sqrt{\pi}\, \delta_{n,m}. $$
Generally, 
the Gauss-Hermite quadrature rule has the form 
$$ \textstyle \int\limits_{-\infty}^{\infty} g(s) \, {\mathrm e}^{-s^2}\, {\mathrm d}s  = \sum\limits_{k=1}^N w_{N,k}^{(H)} \, g(t_{N,k}^{(H)}) + R_N^{(H)}(g),
$$
 where $t_{N,k}^{(H)}$, $k=1, \ldots, N$, denote the simple real symmetric roots of the Hermite polynomial $H_N$ of degree $N$
and where the weights $w_{N,k}^{(H)}$ are given by 
\begin{align}\label{wHermite}  \textstyle w_{N,k}^{(H)} := \frac{2^{N-1} \, N! \, \sqrt{\pi}}{N^2 \, [H_{N-1}(t_{N,k}^{(H)})]^2}, 
\qquad k=1, \ldots , N. 
\end{align}
For real-valued functions $g \in C^{2N}({\mathbb R})$ the quadrature error has the representation 
\begin{align}\label{GHerror} \textstyle R_N^{(H)} (g) = \frac{N! \, \sqrt{\pi}}{2^N(2N)!} g^{(2N)}(\xi), 
\end{align}
where $g^{(2N)}(\xi)$ denotes the $(2N)$-th derivative of $g$ at a suitable value $\xi \in (-\infty, \infty)$, see 
\cite{Abramowitz}, Section 25.4.46.
We apply this quadrature rule to (\ref{rep2}) for even $N$ and obtain with $g_x(u) := \frac{1}{\sqrt{\pi}}\cos\big(\sqrt{\frac{2}{\sigma}} ux\big)$
\begin{align}\label{quad2}  \textstyle {\mathrm e}^{-\frac{x^2}{2\sigma}} = \sum\limits_{k=1}^N \frac{w_{N,k}^{(H)}}{\sqrt{\pi}} \,  \cos\Big(\sqrt{\frac{2}{\sigma}} t_{N,k}^{(H)} x\Big)
 + R_N^{(H)}(g_x) = \sum\limits_{k=1}^{\frac{N}{2}} \frac{2 w_{N,k}^{(H)}}{\sqrt{\pi}} \,  \cos\Big(\sqrt{\frac{2}{\sigma}} t_{N,k}^{(H)} x\Big)
 + R_N^{(H)}(g_x),
 \end{align}
which provides an approximation of the Gaussian function ${\mathrm e}^{-\frac{x^2}{2\sigma}}$ by an exponential sum of length $N$, or more precisely a cosine sum of length $\frac{N}{2}$. The pointwise error is bounded by
\begin{align*}
 \textstyle |R_N^{(H)}(g_x)| =  \frac{N! \, \sqrt{\pi}}{2^N(2N)!} |g_x^{(2N)}(\xi_x)| \le  \frac{N! \, \sqrt{\pi}}{2^N(2N)!} \|g_x^{(2N)}\|_{\infty} = 
 \frac{1}{\sqrt{\pi}}\, \frac{N! \, \sqrt{\pi}}{2^N(2N)!} \Big(\frac{2x^2}{\sigma}\Big)^N =  \frac{N!}{(2N)!} \Big(\frac{x^2}{\sigma}\Big)^N,
\end{align*}
where we have used that 
$|g_x^{(2N)}(s)| 
\le \frac{1}{\sqrt{\pi}}  \Big(\frac{2x^2}{\sigma}\Big)^N$. Applying Stirling's formula, using
$\binom{2N}{N} > \frac{2^{2N}}{\sqrt{\frac{\pi}{2}(2N+1)}}$ and $N! >  \sqrt{2\pi N} \, \big(\frac{N}{{\mathrm e}}\big)^N$, we obtain that 
\begin{align} \nonumber
 \textstyle |R_N^{(H)}(g_x)|
&\le  \textstyle \frac{(N!)^2}{(2N)!} \frac{1}{N!} \Big(\frac{x^2}{\sigma}\Big)^N \le \frac{\sqrt{\pi(2N+1)/2}}{2^{2N}} \frac{1}{N!} \Big(\frac{x^2}{\sigma}\Big)^N = \frac{\sqrt{\pi(2N+1)/2}}{N!} \Big(\frac{x^2}{4\sigma}\Big)^N\\
\label{error2}
&\le  \textstyle  \frac{\sqrt{\pi(2N+1)/2}}{\sqrt{2\pi N}} \Big(\frac{x^2 {\mathrm e}}{4N\sigma}\Big)^N 
= \sqrt{ \frac{(2N+1)}{4N}} \Big(\frac{x^2 {\mathrm e}}{4N\sigma}\Big)^N.
\end{align}
 Based on this error estimate, we achieve pointwise geometric decay of the exponential sum approximation (\ref{quad2})
 if $ x^2 \, {\mathrm e}< 4N\sigma$, i.e., 
$|x| < 2 \sqrt{\frac{N \sigma}{{\mathrm e}}}$. If a decay rate of $\rho^{-2N}$ is desired for a fixed $\rho>1$, this is achieved when
$ \frac{x^2 {\mathrm e}}{4N\sigma} \le \frac{1}{\rho^{2}}$ i.e., $ \frac{x^2 {\mathrm e} \rho^{2}}{4N\sigma} \le 1$,
or $|x| \le  \frac{2}{\rho} \sqrt{\frac{N\sigma}{{\mathrm e}}}$.

\begin{remark}
As in Section $2$, we can also derive an estimate for the considered smooth function $g_x$ based on its expansion into Hermite polynomials,
$$ \textstyle g_x(s)=   \frac{{\mathrm e}^{-\frac{x^2}{2\sigma}} }{\sqrt{\pi}} \sum\limits_{n=0}^{\infty} \frac{\Big( \frac{-x^2}{{2\sigma}}\Big)^n}{(2n)!}\, H_{2n}(s), $$
		see \textnormal{\cite{Szego}}, formula $(5.5.7)$. 
 Then we obtain 
\begin{align}
			\nonumber
\textstyle  R_N^{(H)}(g_x) &= \textstyle   
\int\limits_{-\infty}^{\infty} g_x(s) \, {\mathrm e}^{-s^2} \, {\mathrm d} s
			-  \sum\limits_{k=1}^N w_{N,k}^{(H)} \, g_x(t_{N,k}^{(H)})  = 
			{\mathrm e}^{-\frac{x^2}{2\sigma}} 
			- \frac{1}{\sqrt{\pi}} \sum\limits_{k=1}^N w_{N,k}^{(H)} \, 
			\cos\big(\sqrt{\frac{2}{\sigma}} x t_{N,k}^{(H)} \big)   \\
			& \textstyle = \textstyle 
			 \frac{-{\mathrm e}^{-\frac{x^2}{2\sigma}} }{\sqrt{\pi}}  \Big( \frac{-x^2}{2\sigma}\Big)^N
			\sum\limits_{n=0}^{\infty} \frac{1}{(2n+2N)!}\, \Big( \frac{-x^2}{2\sigma}\Big)^{n} \, 
			\sum\limits_{k=1}^N w_{N,k}^{(H)} \, H_{2n+2N} (t_{N,k}^{(H)}) .
			\label{absg}
		\end{align}
Suitable estimation of all terms in $(\ref{absg})$ then leads to an estimate similar to the one given in 
$(\ref{error2})$.	
\end{remark}

Obviously, for every $x \in {\mathbb R}$ the error $|R_N^{(H)}(g_x)|$ decays geometrically once $N$ is 
taken sufficiently large, while for small $N$ a good error decay may not be achieved.
However, we are interested in good approximations by short exponential sums. 
Therefore, we again propose a procedure that yields slightly adapted approximations on different intervals.

\begin{theorem}\label{Happrox}
For a given $\rho>1$, $\sigma>0$, and arbitrary even length $N \in 2{\mathbb N}$, $N \ge 2$, let $b_j := \frac{2j}{\rho} \, \sqrt{\frac{N\sigma}{{\mathrm e}}}$ for $j\in {\mathbb N}_{0}$. 
Then
the function $f_2(x)= {\mathrm e}^{-\frac{x^2}{2\sigma}}$ 
can be approximated by an exponential 
sum of arbitrary even length $N$, where the error  
\begin{align*}
E(x)&:=
 \textstyle   {\mathrm e}^{-\frac{x^2}{2\sigma}}
-  \frac{2}{\sqrt{\pi}} {\mathrm e}^{\frac{b_j^2}{2\sigma} - \frac{b_j |x|}{\sigma}}
\sum\limits_{k=1}^{ \frac{N}{2} } w_{N,k}^{(H)} \!\cos\big( 
\sqrt{\frac{2}{\sigma}} t_{N,k}^{(H)}(|x|-b_j) \big)
\end{align*}
satisfies  for $|x| \in [b_j, b_{j+1}] $,
\begin{align}\label{approx3H}
 |E(x)| &\le  \textstyle  \sqrt{\frac{2N+1}{4N}}  \big(\frac{(|x|-b_j)^2 {\mathrm e}}{4 N \sigma} \big)^N \!  {\mathrm e}^{-\frac{b_j 
(2|x|-b_j)}{2\sigma}} 
\le 
\textstyle  \sqrt{\frac{2N+1}{4N}}  \,  {\mathrm e}^{-\frac{b_j^2}{2\sigma}}\, \rho^{-2N}  .
\end{align}
%where $b_0:=0$ and  $b_j := \frac{2j}{\rho} \, \sqrt{\frac{N\sigma}{{\mathrm e}}}$ for $j\in {\mathbb N}$.
 \end{theorem}

\begin{proof} Let $|x| \in [b_j, b_{j+1}]$ and $x':=|x|-b_j \in [0, b_1] = 
[0, \frac{2}{\rho} \, \sqrt{\frac{N\sigma}{{\mathrm e}}}]$.
Then we obtain from (\ref{rep2}) and (\ref{quad2})  with 
$g_{x'}(u)= \frac{1}{\sqrt{\pi}}\cos\big(\sqrt{\frac{2}{\sigma}} ux'\big)$
\begin{align*}
{\mathrm e}^{-\frac{x^2}{2\sigma}} &=  \textstyle {\mathrm e}^{-\frac{(x'+b_j)^2}{2\sigma}} 
=  {\mathrm e}^{-\frac{b_j (2x'+b_j)}{2\sigma}} \, {\mathrm e}^{-\frac{x'^2}{2\sigma}} 
= {\mathrm e}^{-\frac{b_j (2x'+b_j)}{2\sigma}}  \,  \frac{1}{\sqrt{\pi}}  \int\limits_{-\infty}^{\infty}  {\mathrm e}^{-u^2} \, \cos\big(\sqrt{\frac{2}{\sigma}} ux'\big) \, {\mathrm d} u \\
&= \textstyle  {\mathrm e}^{-\frac{b_j (2x'+b_j)}{2\sigma}} 
 \Big( \sum\limits_{k=1}^{\frac{N}{2}} \frac{2w_{N,k}^{(H)}}{\sqrt{\pi}} \, \cos\big(\sqrt{\frac{2}{\sigma}} t_{N,k}^{(H)} x'\big)   + R_N^{(H)}(g_{x'}) \Big) \\
 & = \textstyle    \frac{2}{\sqrt{\pi}} {\mathrm e}^{ - \frac{b_j(2 |x|-b_j)}{2\sigma}}
\sum\limits_{k=1}^{ N/2 } \!\! w_{N,k}^{(H)} \!\cos\big( 
\sqrt{\frac{2}{\sigma}}  t_{N,k}^{(H)}(|x|- b_{j}) \big) +  {\mathrm e}^{-\frac{b_j (2|x|- b_j)}{2\sigma}} \, R_N^{(H)}(g_{x'}).
\end{align*}
For the error $ {\mathrm e}^{-\frac{b_j (2x'+b_j)}{2\sigma}} \, R_N^{(H)}(g_{x'})  $ 
we conclude from (\ref{error2}) with $b_1= \frac{2}{\rho} \, \sqrt{\frac{N\sigma}{{\mathrm e}}}$ the estimate
\begin{align*}|  \textstyle  {\mathrm e}^{-\frac{b_j (2x'+b_j)}{2\sigma}} \, R_N^{(H)}(g_{x'}) | &\le   \textstyle \sqrt{\frac{2N+1}{4N}} \, \big(\frac{x'^2 {\mathrm e}}{4 N \sigma} \big)^N \,  {\mathrm e}^{-\frac{b_j (2x'+b_j)}{2\sigma}}
\le \textstyle \sqrt{\frac{2N+1}{4N}} \, \big(\frac{b_1^2 {\mathrm e}}{4 N \sigma} \big)^N \,  {\mathrm e}^{-\frac{b_j^2}{2\sigma}} \\
&=  \textstyle \sqrt{\frac{2N+1}{4N}} \, \big(\frac{4 N \sigma}{ 4 N \sigma \rho^2} \big)^N \,  {\mathrm e}^{-\frac{b_j^2}{2\sigma}}
= \sqrt{\frac{2N+1}{4N}}  \,  {\mathrm e}^{-\frac{b_j^2}{2\sigma}}\, \rho^{-2N}.
\end{align*}
In particular, by $x':=|x|-b_j$, 
$$ |  \textstyle  {\mathrm e}^{-\frac{b_j (2x'+b_j)}{2\sigma}} \, R_N^{(H)}(g_{x'}) | \le  
\textstyle \sqrt{\frac{2N+1}{4N}} \, \big(\frac{(|x|-b_j)^2 {\mathrm e}}{4 N \sigma} \big)^N \,  {\mathrm e}^{-\frac{b_j 
(2|x|-b_j)}{2\sigma}}. $$
Thus, formula (\ref{approx3H}) follows. The error decay is even stronger than $\rho^{-2N}$ for  $j>0$, namely
$ {\mathrm e}^{-\frac{b_j^2}{2\sigma}}\, \rho^{-2N} = \big({\mathrm e}^{\frac{j^2}{\rho^2 {\mathrm e}}} \, \rho\big)^{-2N}$.
\end{proof}

\subsection{Approximation on consecutive intervals and numerical experiments}

Again we can show that the approximations of ${\mathrm e}^{-\frac{x^2}{2\sigma}}$ at different intervals are very closely related, so that 
we obtain (up to a scaling) the same parameters for the exponential sum approximation if a suitable transformation to the inspected interval is applied. Moreover, the same relative error occurs at every interval.

\begin{theorem}
Let  $\rho>1$, $N \in 2{\mathbb N}$, $N\ge 2$,  and  $\sigma>0$ be given and let 
$b_j:=\frac{2j}{\rho} \, \sqrt{\frac{N\sigma}{{\mathrm e}}}$ for $j \in {\mathbb N}_0$.
Then, for $|x| \in [b_j, b_{j+1}]$ and $\tilde{x} := \frac{|x| - b_j}{b_{j+1}-b_j} \in [0, 1]$ 
the exponential sum approximation of  ${\mathrm e}^{-x^2/(2\sigma)}$ is of the form 
\begin{align*} \textstyle 
 \frac{2 {\mathrm e}^{\frac{b_j^2}{2\sigma} - \frac{b_j |x|}{\sigma}}}{\sqrt{\pi}} 
\sum\limits_{k=1}^{\frac{N}{2}} \!\! w_{N,k}^{(H)} \!\cos\big( 
\sqrt{\frac{2}{\sigma}}  t_{N,k}^{(H)}( |x|-b_j) \big)
= \textstyle  \frac{2{\mathrm e}^{-\frac{b_j}{2\sigma} (b_j+ 2b_1 \tilde{x})}}{\sqrt{\pi}}
\sum\limits_{k=1}^{\frac{N}{2}} \!\! w_{N,k}^{(H)}  \cos\big( 
\sqrt{\frac{2}{\sigma}}  t_{N,k}^{(H)}b_1\tilde{x} \big).
\end{align*}
In particular, the relative error 
\begin{align*} \textstyle 
 \frac{{\mathrm e}^{-\frac{x^2}{2\sigma}} -  \frac{2 {\mathrm e}^{\frac{b_j^2}{2\sigma} - \frac{b_j |x|}{\sigma}}}{\sqrt{\pi}} 
\sum\limits_{k=1}^{\frac{N}{2}} \!\! w_{N,k}^{(H)} \!\cos\big( 
\sqrt{\frac{2}{\sigma}}  t_{N,k}^{(H)}(|x|-b_j) \big)}{{\mathrm e}^{-\frac{x^2}{2\sigma}}}
=  \textstyle  1- \frac{2}{\sqrt{\pi}} \displaystyle  {\mathrm e}^{\frac{2N \tilde{x}^2}{{\mathrm e} \rho^2}} 
\textstyle \sum\limits_{k=1}^{\frac{N}{2}}  w_{N,k}^{(H)} \cos\big(\sqrt{\frac{2}{\sigma}} t_{N,k}^{(H)}b_1 \tilde{x}\big)
\end{align*}
does not  depend on the interval $[b_j, b_{j+1}]$ and is bounded by $\sqrt{\frac{2N+1}{4N}} {\mathrm e}^{\frac{b_1^2}{2\sigma}} \, \rho^{-2N}$.
\end{theorem} 

\begin{proof} Since ${\mathrm e}^{-\frac{x^2}{2\sigma}} $ is even, we only consider $x \in [0, \infty)$.
For $x=b_j + \tilde{x}(b_{j+1}-b_j) =b_j + \tilde{x}b_1 \ge0$ with $\tilde{x} \in [0,1]$  and 
$b_{j+1}-b_j =
\frac{2}{\rho} \, \sqrt{\frac{N\sigma}{{\mathrm e}}} = b_1$ we find $x-b_j = b_1 \tilde{x}$. Then the exponential sum approximation takes the form  
\begin{align*} \textstyle 
 \frac{2 {\mathrm e}^{\frac{b_j^2}{2\sigma} - \frac{b_j |x|}{\sigma}}}{\sqrt{\pi}} 
\sum\limits_{k=1}^{\frac{N}{2}} \!\! w_{N,k}^{(H)} \!\cos\big( 
\sqrt{\frac{2}{\sigma}}  t_{N,k}^{(H)}( |x|-b_j) \big)
& \textstyle = \frac{2}{\sqrt{\pi}}  {\mathrm e}^{-\frac{b_j}{2\sigma} (b_j+ 2b_1 \tilde{x})} 
\sum\limits_{k=1}^{\frac{N}{2}} \!\! w_{N,k}^{(H)} \cos(\sqrt{\frac{2}{\sigma}} t_{N,k}^{(H)} b_1 \tilde{x}).
\end{align*}
For the relative error it follows that 
\begin{align*}
\textstyle 
1-{\mathrm e}^{\frac{x^2}{2\sigma}}  \frac{2 {\mathrm e}^{\frac{b_j^2}{2\sigma} - \frac{b_j |x|}{\sigma}}}{\sqrt{\pi}} 
\sum\limits_{k=1}^{\frac{N}{2}} \!\! w_{N,k}^{(H)} \!\cos\big( 
\sqrt{\frac{2}{\sigma}}  t_{N,k}^{(H)}( |x|-b_j) \big)
=  1-\frac{2}{\sqrt{\pi}}  {\mathrm e}^{\frac{(x-b_j)^2}{2\sigma}} 
\sum\limits_{k=1}^{\frac{N}{2}} \!\! w_{N,k}^{(H)} \cos(\sqrt{\frac{2}{\sigma}}  t_{N,k}^{(H)} b_1 \tilde{x}),
\end{align*}
where $\frac{x^2}{2\sigma}+\frac{b_j^2}{2\sigma} - \frac{b_j x}{\sigma}= \frac{(x-b_j)^2}{2\sigma} = \frac{b_1^2\tilde{x}^2}{2\sigma} = \frac{2N \tilde{x}^2}{{\mathrm e} \rho^2}$. 
For  the relative error bound we conclude from  (\ref{approx3H}) 
that 
\begin{align*}
{\mathrm e}^{x^2/(2\sigma)} \, |E(x)| &\le \textstyle 
{\mathrm e}^{\frac{x^2}{2\sigma}} \sqrt{\frac{2N+1}{4N}} \, 
{\mathrm e}^{\frac{-b_j(2|x|-b_j)}{2\sigma}}\, \rho^{-2N}
=  \sqrt{\frac{2N+1}{4N}} \, {\mathrm e}^{\frac{(x-b_j)^2}{2\sigma}}\, \rho^{-2N} \le 
 \sqrt{\frac{2N+1}{4N}} \, {\mathrm e}^{\frac{b_1^2}{2\sigma}}\, \rho^{-2N}.
\end{align*}
\end{proof}

\begin{remark}
We can also apply a different exponential sum approximation for ${\mathrm e}^{-x^{2}/(2\sigma)}$ on the first interval $[0, b]$, $b >0$. As shown in \textnormal{\cite{Derevianko26}}, the differential approximation performs significantly better on small intervals $[0, b]$ than other algorithms based on Prony's method, and also better than our explicit exponential sum representation of ${\mathrm e}^{-x^{2}/(2\sigma)}$ in Theorem $\ref{Happrox}$.
Compared to our approximation formula, the differential approximation method uses similar frequency parameters (scaled zeros of Hermite polynomials) but computes different weights numerically based on minimization in a weighted $L^{2}$-norm. 

Assume that we know an approximation 
$$ \textstyle  \big| {\mathrm e}^{-\frac{x^{2}}{2\sigma}} - \sum\limits_{k=1}^N c_k \, {\mathrm e}^{-T_k \, x} \big| < \epsilon, \qquad x \in [0, {b}] $$
with $c_{k}, \, T_{k} \in {\mathbb C}$.
Let now $b_{j}:= j \, b$ for $j \in {\mathbb  N}$. Then we obtain for $x\in [b_{j}, b_{j+1})$ with $x=b_{j}+x'$
\begin{align*}
 \textstyle  
 \big| {\mathrm e}^{-\frac{x^{2}}{2\sigma}}- {\mathrm e}^{-\frac{b_{j}^{2} +2b_{j}x'}{2\sigma}}  
 \sum\limits_{k=1}^N c_k \, {\mathrm e}^{-T_k \, x'}
 \big| <  {\mathrm e}^{-\frac{b_{j}^{2} +2b_{j}x'}{2\sigma}}\epsilon <  {\mathrm e}^{-\frac{b_{j}^{2}}{2\sigma}}\epsilon 
 \qquad x \in [{b}_j, {b}_{j+1}].
 \end{align*}
 \end{remark}

 \begin{figure}[htb]
\begin{center}
	\includegraphics[scale=0.35]{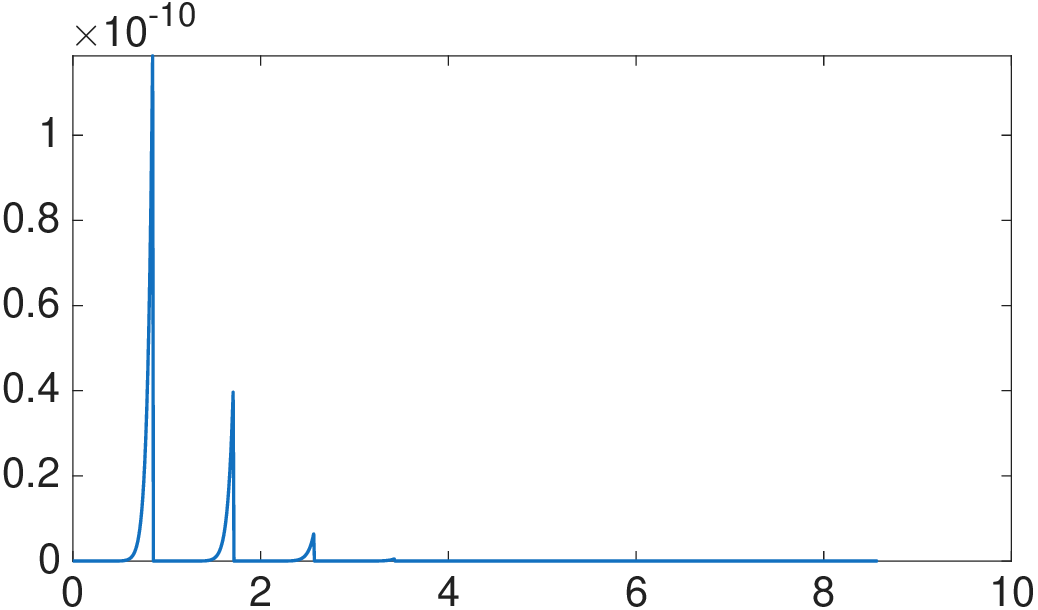} 
	\includegraphics[scale=0.35]{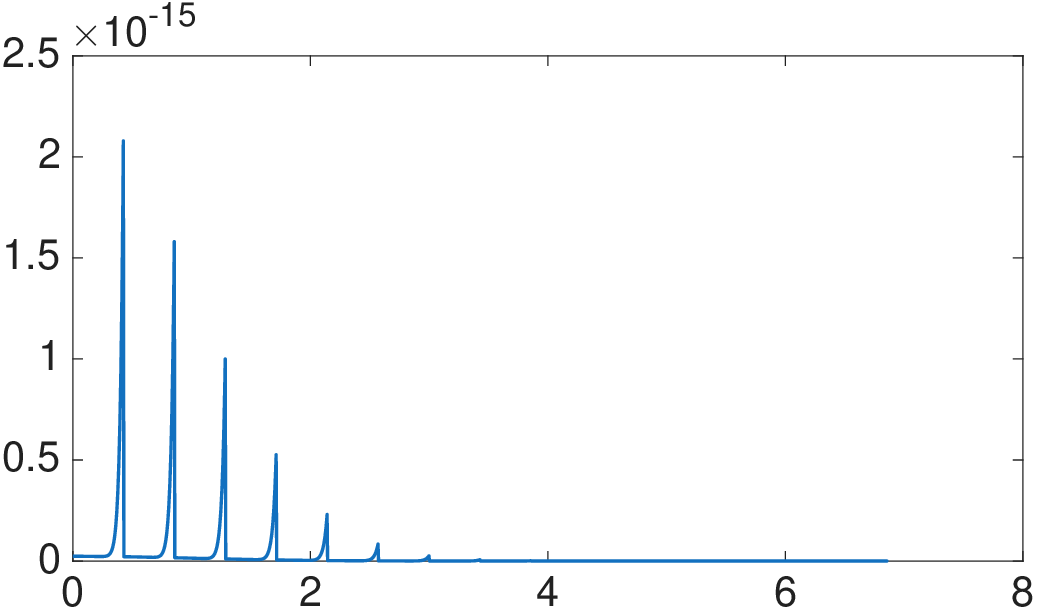} \\
		\includegraphics[scale=0.35]{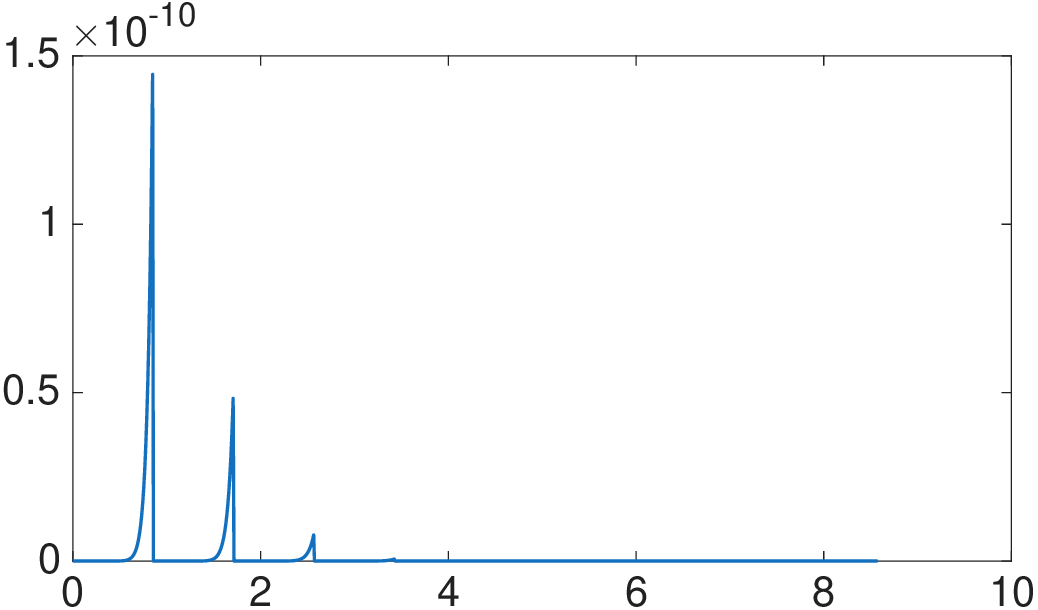}
	\includegraphics[scale=0.35]{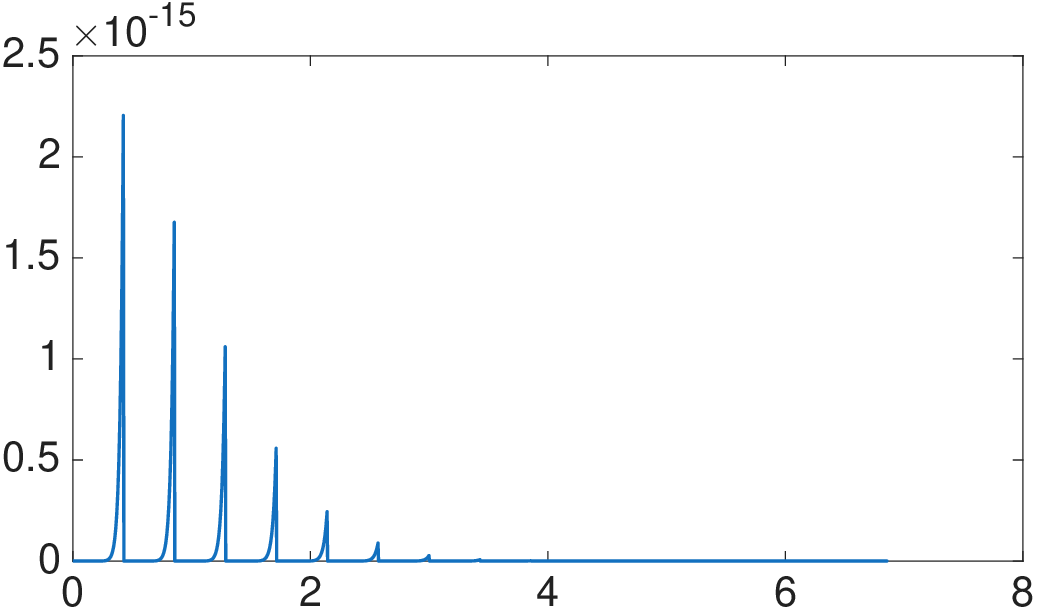}\\
	\includegraphics[scale=0.35]{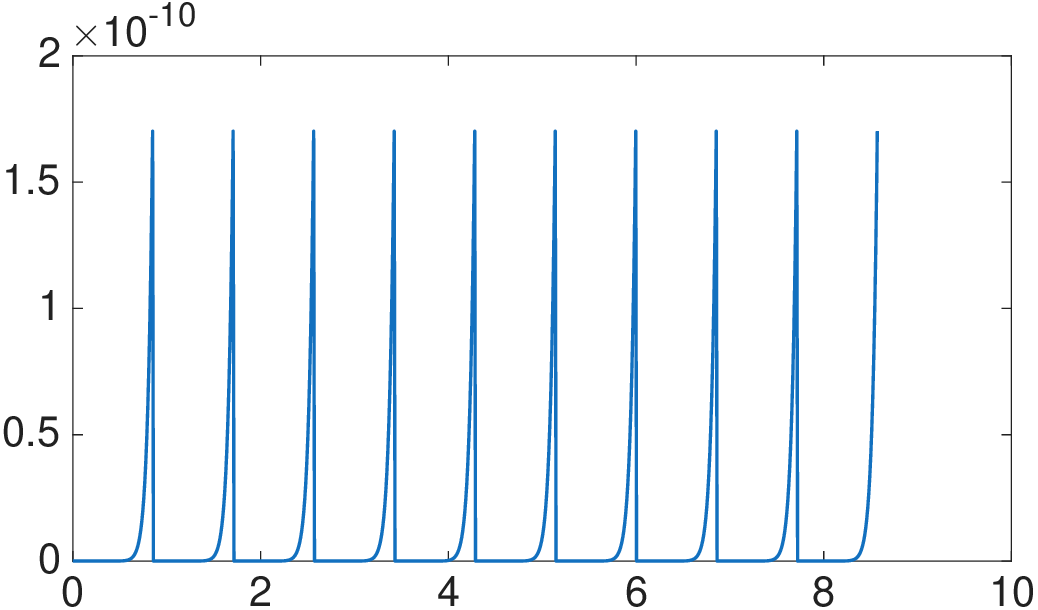} 
         \includegraphics[scale=0.35]{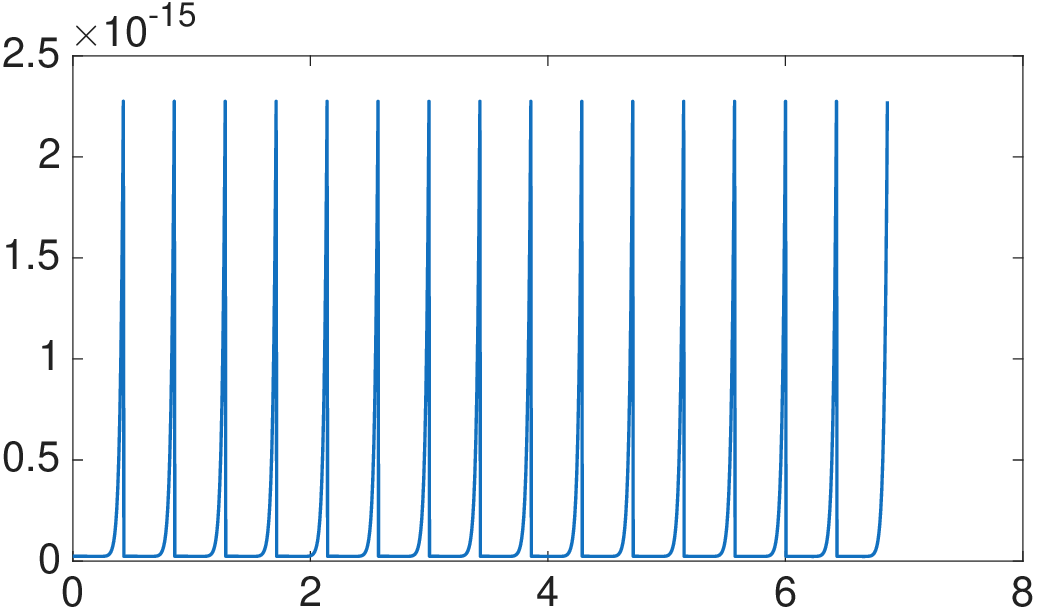} \\
\end{center}
\captionsetup{aboveskip=0pt, justification=justified, labelfont=small, labelsep=space, labelfont=bf}
\caption{\small
Computational error of $\big|{\mathrm e}^{-\frac{x^2}{2}} - 
 \frac{2 {\mathrm e}^{\frac{b_j^2}{2\sigma} - \frac{b_j |x|}{\sigma}}}{\sqrt{\pi}} 
\sum\limits_{k=1}^{\frac{N}{2}} \!\! w_{N,k}^{(H)} \!\cos\big( 
\sqrt{\frac{2}{\sigma}}  t_{N,k}^{(H)}( |x|-b_j) \big)\big|$
for $N=8$,  $\rho=4$ and $x \in [0,8.57)$ (left column), and for $N=8$,  $\rho=8$ and $x \in [0,6.86)$ (right column).
Top: absolute errors for $\rho=4$ (left) and $\rho = 8$ (right). 
Middle:  error  bounds according to (\ref{approx3H})  for $\rho=4$ (left) and $\rho=8$ (right).
Bottom: relative errors for $\rho=4$ (left) and $\rho = 8$ (right). 
}
\label{fig3}
\end{figure}

 \begin{table}[htb]
\scriptsize
\caption{\small Nodes and weights for Gauss-Hermite quadrature for $N=8$.}
\begin{center}
\begin{tabular}{|c|l|l|}
\hline
$k$ & nodes $t_{8,k}^{(H)}$ &   weights $w_{8,k}^{(H)}$ \\
\hline
1 &  $\pm$ 2.93063742025724401922350270524360 & 0.000199604072211367619206090452544096456225712638 \\
2 &  $\pm$ 1.98165675669584292585463063976930 & 0.017077983007413475456203056436445678180510424856 \\
3 &  $\pm$ 1.15719371244678019472076577906310  &0.207802325814891879543258620285700557593576026897 \\
4 &  $\pm$ 0.38118699020732211685471888558369 & 0.661147012558241291030415974495882259168462563670 \\
\hline
\end{tabular}
\end{center}
\label{tab3}
\end{table}

\begin{example}
We approximate ${\mathrm e}^{-\frac{x^2}{2\sigma}}$ for $\sigma = 1$
employing an exponential sum of length $N=8$. Note that ${\mathrm e}^{-\frac{x^2}{2\sigma}}$ is even and already smaller than $1.2665 \cdot 10^{-14}$ for $|x| \ge 8$ and smaller than $2.5768\cdot 10^{-18}$ for $x\ge 9$. Working in double-precision arithmetic, we only need to consider intervals $[0, b]$ with $b < 9$.
 Applying Theorem $\ref{Happrox}$ for $\rho= 4$, we choose
$b_j=  \sqrt{\frac{2}{{\mathrm e}}} j \approx 0.8578  \cdot j$ and obtain in
 this case from $(\ref{approx3H})$ 
\begin{align*}  \textstyle \Big| {\mathrm e}^{-\frac{x^2}{2}} - 
\frac{2}{\sqrt{\pi}} {\mathrm e}^{j\big(\frac{j}{{\mathrm e}} - 
\sqrt{\frac{2}{{\mathrm e}}} |x| \big)}
\sum\limits_{k=1}^4 w_{N,k}^{(H)} \, 
\cos\big( t_{N,k}^{(H)}\big(\sqrt{2} |x| -\frac{2j}{\sqrt{\mathrm e}}\big)\big) \Big|
&\le \textstyle  \sqrt{\frac{17}{32}}  
\frac{(|x|- \sqrt{\frac{2}{{\mathrm e}}}j)^{16}}{2^{40}} \,  {\mathrm e}^8\, 
{\mathrm e}^{-\sqrt{\frac{2}{{\mathrm e}}}j|x|+ \frac{j^2}{{\mathrm e}}} \\
& \le \textstyle 
\sqrt{\frac{17}{32}} \,  {\mathrm e}^{-\frac{j^2}{{\mathrm e}}} \, 4^{-16}
\end{align*}
for all 
$|x| \in  \big[ b_j, \,b_{j+1} \big]$.
The computed absolute error and relative error 
are illustrated in Figure $\ref{fig3}$ (left column) for $10$ intervals, i.e., for $x \in [0, 8.57]$.  Furthermore, we provide the theoretical error bound $ \sqrt{\frac{17}{32}}  \frac{(|x|- \sqrt{\frac{2}{{\mathrm e}}}j)^{16}}{2^{40}} \, {\mathrm e}^{-\sqrt{\frac{2}{{\mathrm e}}}j|x|+ \frac{j^2}{{\mathrm e}}+8}$ from $(\ref{approx3H})$
(left column, middle).
Similarly, we show the numerical results for $N=8$, $\rho=8$ and $b_j= \frac{j}{\sqrt{2{\mathrm e}}}$, where we use $16$ intervals. The computed absolute error, relative error, and error bound  
are illustrated in Figure $\ref{fig3}$ (right column) for $x \in [0, 6.86]$.
The nodes and weights for the Gauss-Hermite quadrature for $N=8$ are given in Table 
$\ref{tab3}$ with at least $32$ exact digits.
\end{example}

\begin{figure}[bht]
\begin{center}
	\includegraphics[scale=0.35]{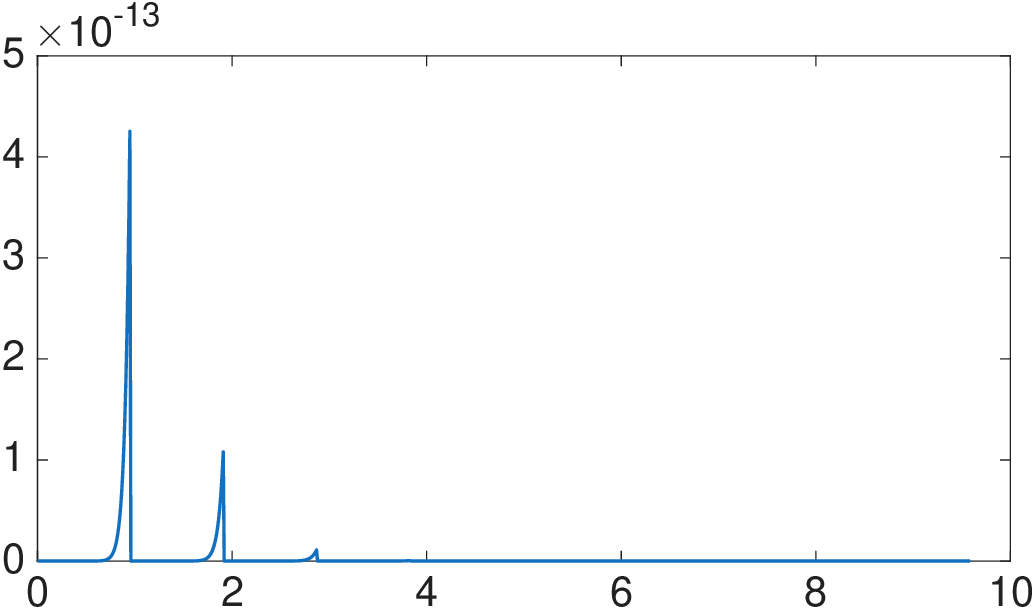} 
	\includegraphics[scale=0.35]{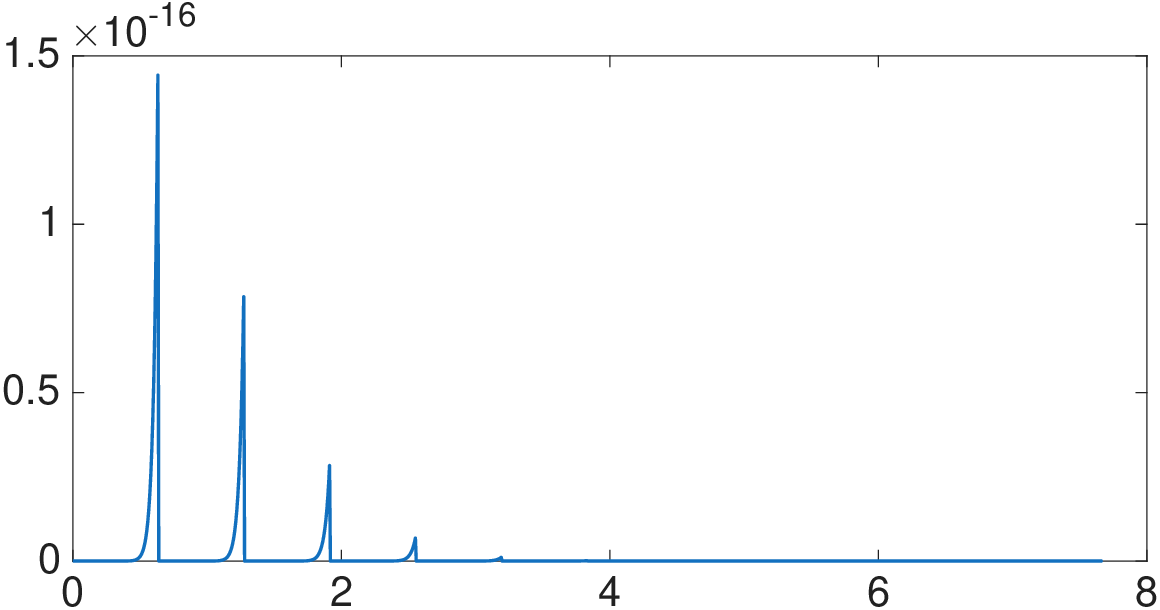} \\
	\includegraphics[scale=0.35]{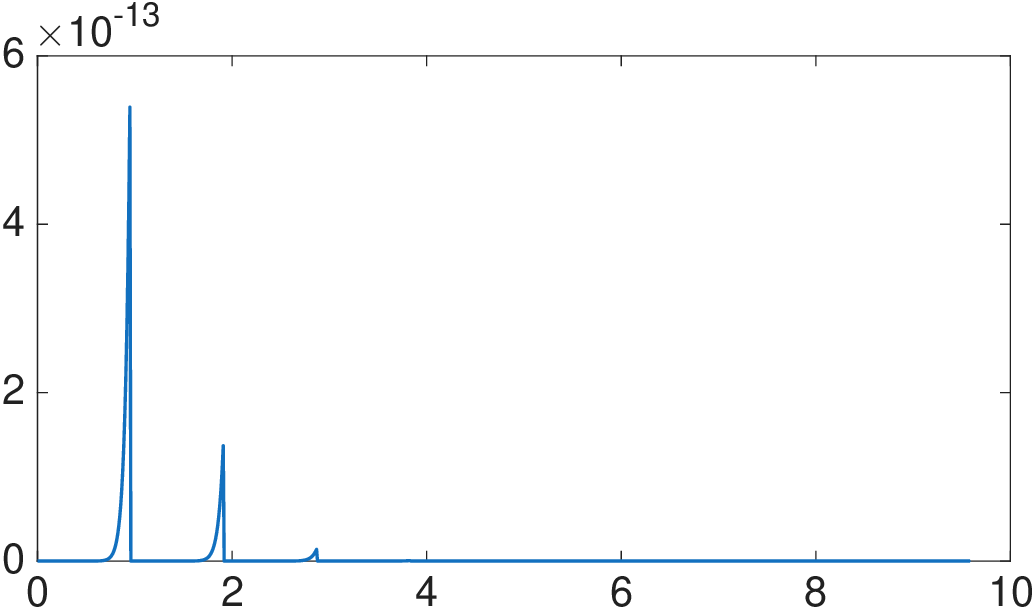}
	\includegraphics[scale=0.35]{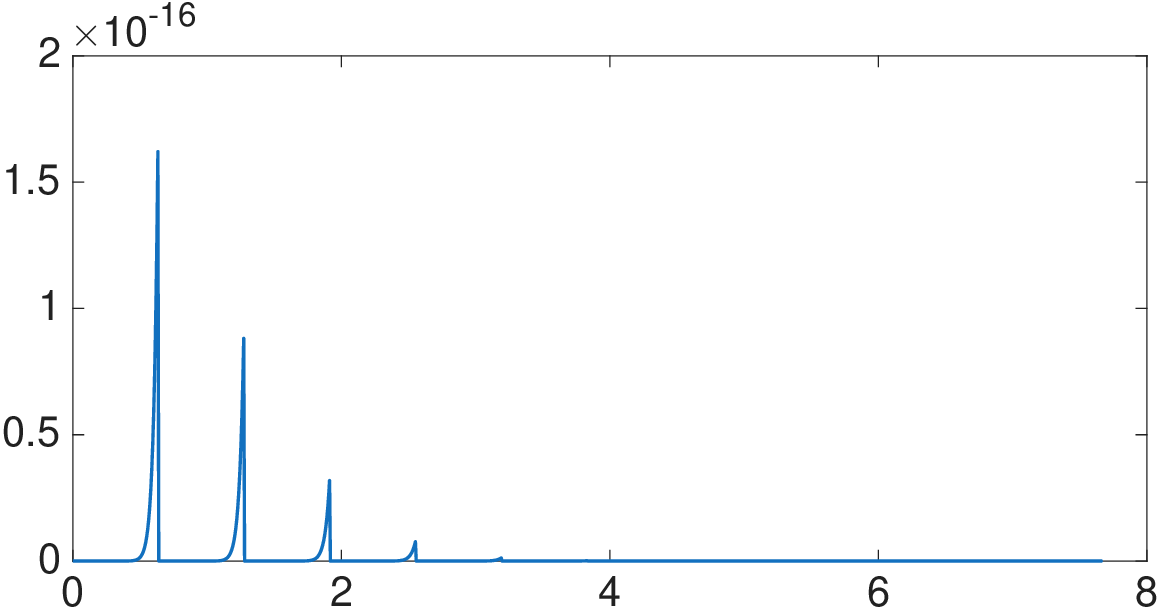}\\
		\includegraphics[scale=0.35]{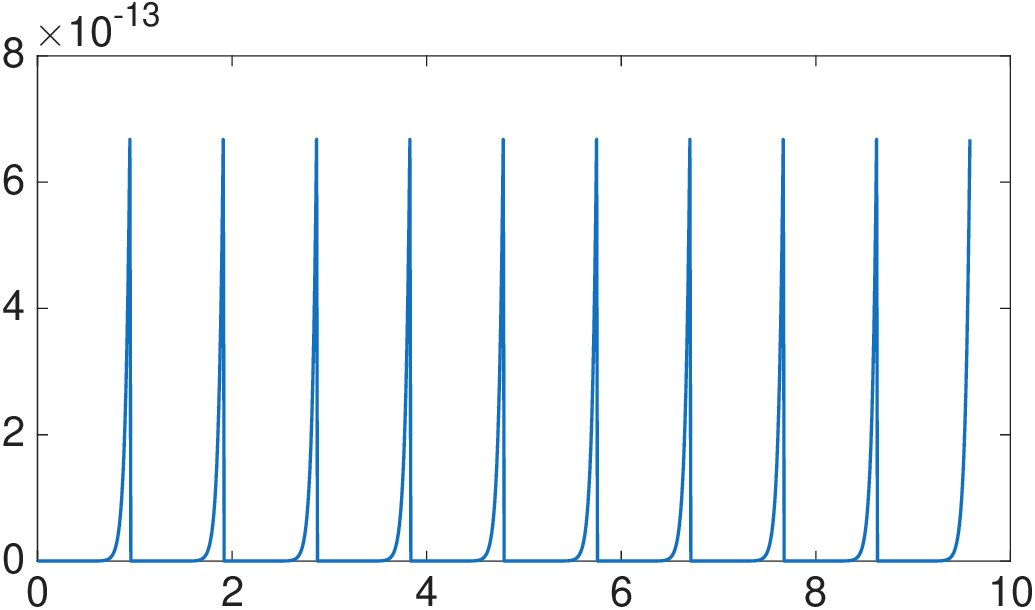} 
         \includegraphics[scale=0.35]{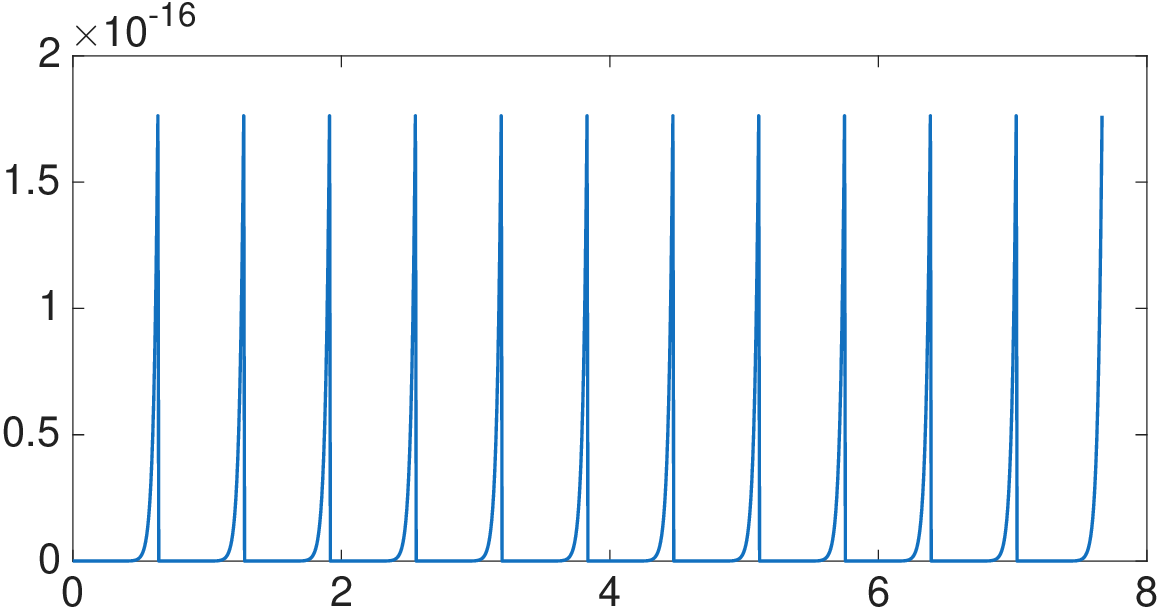} 
	\end{center}
\captionsetup{aboveskip=0pt, justification=justified, labelfont=small, labelsep=space, labelfont=bf}
\caption{\small
Computational error of $\big|{\mathrm e}^{-\frac{x^2}{2}} -
 \frac{2 {\mathrm e}^{\frac{b_j^2}{2\sigma} - \frac{b_j |x|}{\sigma}}}{\sqrt{\pi}} 
\sum\limits_{k=1}^{\frac{N}{2}} \!\! w_{N,k}^{(H)} \!\cos\big( 
\sqrt{\frac{2}{\sigma}}  t_{N,k}^{(H)}( |x|-b_j) \big)\big|$
for $N=10$,  $\rho=4$ and $x \in [0,9.59)$, 10  intervals (left column), and for $N=10$,  $\rho=6$, and 
$x \in [0,7.67)$
12 intervals (right column).
Top: absolute errors for $\rho=4$ (left) and $\rho = 6$ (right). 
Middle: error  bounds according to (\ref{approx3H})  for $\rho=4$ (left) and $\rho=6$ (right).
Bottom: relative errors for $\rho=4$ (left) and $\rho = 6$ (right). 
}
\label{fig4}
\end{figure}

\begin{example}
Next, we approximate ${\mathrm e}^{-\frac{x^2}{2\sigma}}$ for $\sigma = 1$
employing an exponential sum of length $N=10$. 
 Applying Theorem $\ref{Happrox}$ for $\rho= 4$, we choose
$b_j=  \frac{j}{2} \sqrt{\frac{10}{{\mathrm e}}} \approx 0.9590  \cdot j$.
The computed absolute error and relative error  
are illustrated in Figure $\ref{fig4}$ (left column) for $10$ intervals, i.e., for $x \in [0, 9.59]$.  
Furthermore, we provide the error bound as given in $(\ref{approx3H})$
(left column, middle).
Similarly, we show the numerical results for $N=10$, 
$\rho=6$ and $b_j= \frac{2j}{6}  \sqrt{\frac{10}{{\mathrm e}}} \approx 0.6393 \, j$,
where we use $12$ intervals. The computed absolute error, relative error, and error bound 
are illustrated in Figure $\ref{fig4}$ (right column) for $x \in [0, 7.67]$.
Nodes and weights for the Gauss-Hermite quadrature with at least $32$ exact digits are given in Table $\ref{tab4}$.

 \begin{table}[tbh]
\scriptsize
\caption{\small Nodes and weights for Gauss-Hermite quadrature for $N=10$.}
\begin{center}
\begin{tabular}{|c|l|l|}
\hline
$k$ & nodes $t_{10,k}^{(H)}$ &   weights $w_{10,k}^{(H)}$ \\
\hline
1 & $\pm$ 3.43615911883773760332672549431912  &
0.000007640432855232620629159367859595222 
\\
2 & $\pm$ 2.53273167423278979640896079775479  &
0.001343645746781232692201565585845913870  
\\
3 & $\pm$  1.75668364929988177345140122010616 &
0.033874394455481063136164731277585973698 
\\
4 &  $\pm$  1.03661082978951365417749191675921  &
0.240138611082314686416523295005861395370 \\
5 &  $\pm$  0.34290132722370460878916502555726  &
0.610862633735325798783564990433419713239  
\\
\hline
\end{tabular}
\end{center}
\label{tab4}
\end{table}
\end{example}

\subsection{Application to represent the error function with high precision}
\label{ApproxErf}
Our approximation of ${\mathrm e}^{-\frac{x^2}{2\sigma}}$ by short exponential sums can now be directly applied to derive a short exponential sum representation 
of the error function 
\begin{equation*}
	\textstyle \mathrm{erf}(x) = \frac{2}{\sqrt{\pi}} \int\limits_0^x {\mathrm e}^{-t^2} \, {\mathrm d} t
\end{equation*}
with high precision. We set $\sigma= \frac{1}{2}$ and $N$ even.
Then we find from Theorem \ref{Happrox} for $x\in[b_j,b_{j+1}]$
\begin{align*}
	& \textstyle \int\limits_{b_j}^{x} {\mathrm e}^{-t^2} {\mathrm d} t  \textstyle \approx  \int\limits_{b_j}^{x} \frac{2}{\sqrt{\pi}} 
	{\mathrm e}^{b_j^2- 2 b_j t}
	\sum\limits_{k=1}^{\frac{N}{2}}  w_{N,k}^{(H)} \cos\big( 
	2 t_{N,k}^{(H)}(t- b_j) \big)  {\mathrm d}t \\
	&= \textstyle \frac{2}{\sqrt{\pi}}\sum\limits_{k=1}^{\frac{N}{2}} w_{N,k}^{(H)}  
	{\mathrm e}^{-b_j^2} \int\limits_{b_j}^{x} 
	{\mathrm e}^{2 b_j (b_j- t)} \cos\big(2 t_{N,k}^{(H)}(t- b_j) \big)  {\mathrm d}t\\
	&=\textstyle \frac{2}{\sqrt{\pi}}\sum\limits_{k=1}^{\frac{N}{2}}  w_{N,k}^{(H)}  
	{\mathrm e}^{-b_j^2} \int\limits_{0}^{x-b_j} 
	{\mathrm e}^{-2 b_j t'} \cos\big(2 t_{N,k}^{(H)}t' \big)  {\mathrm d}t'\\
	&= \textstyle \frac{1}{\sqrt{\pi}} \sum\limits_{k=1}^{\frac{N}{2}}  w_{N,k}^{(H)}  
	\frac{{\mathrm e}^{-b_j^2}}{b_j^2 +  (t_{N,k}^{(H)})^2} \Big(b_j +{\mathrm e}^{-2b_j(x-b_j)}
	\big(- b_j\cos(2 t_{N,k}^{(H)} (x-b_j))%\\
	%&\qquad\qquad\qquad\qquad\qquad\qquad\qquad\qquad\qquad\qquad
	+  t_{N,k}^{(H)} \sin(2 t_{N,k}^{(H)}(x-b_j) \big)  \Big),
\end{align*}
and therefore
\begin{align*}
	 \mathrm{erf}(x) &=  \textstyle \mathrm{erf}(b_j)+ \frac{2}{\sqrt{\pi}}\int\limits_{b_j}^{x} {\mathrm e}^{-t^2} {\mathrm d} t \\
	&\approx \textstyle \mathrm{erf}(b_j)+\frac{2}{\pi}\sum\limits_{k=1}^{\frac{N}{2}}  w_{N,k}^{(H)}  
	\frac{{\mathrm e}^{-b_j^2}}{b_j^2 +  (t_{N,k}^{(H)})^2} \Big(b_j +{\mathrm e}^{-2b_j(x-b_j)}
	\big(- b_j\cos(2 t_{N,k}^{(H)} (x-b_j))\\
	&\qquad\qquad\qquad\qquad\qquad\qquad\qquad\qquad\qquad\qquad\qquad\quad
	+  t_{N,k}^{(H)} \sin(2 t_{N,k}^{(H)}(x-b_j) \big)  \Big).
\end{align*}
Using the same notation as in Theorem \ref{Happrox}, we obtain the following error estimate for this approximation.
\begin{Corollary}\label{ErfApproximation}
	Assume that we have approximated ${\mathrm e}^{-x^2}$ by an exponential sum of length 
	$N\in2\N$ with $\rho>1$ as given in Theorem $\ref{Happrox}$. Then for $x\in[b_j,b_{j+1}]$ the  error  
	\begin{align*}
		E(x) &= \textstyle \mathrm{erf}(x) - \mathrm{erf}(b_j)\\
		&\textstyle \quad -\frac{2}{\pi}\sum\limits_{k=1}^{\frac{N}{2}}    
		\frac{w_{N,k}^{(H)} {\mathrm e}^{-b_j^2}}{b_j^2 +  (t_{N,k}^{(H)})^2} \Big(b_j +{\mathrm e}^{-2b_j(x-b_j)}
		\big(\!\!-b_j\cos(2 t_{N,k}^{(H)} (x-b_j))%\\
		%&\qquad\qquad\qquad\qquad\qquad\qquad\qquad\qquad\qquad\quad
		\!+  t_{N,k}^{(H)} \sin(2 t_{N,k}^{(H)}(x-b_j) \big)  \Big)
	\end{align*} 
	satisfies 
	\begin{align}\label{erf}  |E(x)|  \le \textstyle  \frac{2}{\sqrt{\pi}} \frac{1}{\sqrt{2{\mathrm e} (2N+1)}} 
		{\mathrm e}^{-b_j^2} \rho^{-(2N+1)}.
	\end{align}
\end{Corollary} 
\begin{proof} 
	Formula (\ref{approx3H}) yields for $x\in[b_j,b_{j+1}]$
	\begin{align*} &   \textstyle  \Big| \int\limits_{b_j}^{x} {\mathrm e}^{-t^2} \, {\mathrm d}t 
		- \int\limits_{b_j}^{x} \frac{2}{\sqrt{\pi}} 
		{\mathrm e}^{b_j^2- 2b_j t}
		\sum\limits_{k=1}^{\frac{N}{2}}  w_{N,k}^{(H)} \cos\big( 
		2 t_{N,k}^{(H)}(t- b_j) \big)  {\mathrm d}t  \Big| \\
		& \le   \textstyle 
		\sqrt{\frac{2N+1}{4N}}   \int\limits_{b_j}^{b_{j+1}} \Big(\frac{(t-b_j)^2 
			{\mathrm e}}{2N}\Big)^N {\mathrm e}^{-2b_j (t - b_j) - b_j^2} {\mathrm d} t\\
		&= \textstyle \sqrt{\frac{2N+1}{4N}}  \Big(\frac{{\mathrm e}}{2N}\Big)^N {\mathrm e}^{-b_j^2} \int\limits_{0}^{b_1} t^{2N} {\mathrm e}^{-2b_j t} {\mathrm d} t
		\\
		& \le \textstyle   \sqrt{\frac{2N+1}{4N}}  \Big(\frac{{\mathrm e}}{2N}\Big)^N {\mathrm e}^{-b_j^2} \frac{b_1^{2N+1}}{2N+1}
		= \frac{1}{2N+1} \sqrt{\frac{2N+1}{4N}}  \Big(\frac{{\mathrm e}}{2N}\Big)^N {\mathrm e}^{-b_j^2} \Big(\frac{1}{\rho} \sqrt{\frac{2N}{{\mathrm e}}}\Big)^{2N+1}\\
		&=   \textstyle  \frac{1}{\sqrt{2{\mathrm e} (2N+1)}} 
		{\mathrm e}^{-b_j^2} \rho^{-(2N+1)}.
	\end{align*}
	Thus (\ref{erf}) follows.
\end{proof}

\begin{remark}
The approximation of $\mathrm{erf}(x)$ for $x \in [b_{j}, b_{j+1})$ in Corollary $\ref{ErfApproximation}$ can be rewritten as the exponential sum
\begin{align*} 
& \textstyle 	
\mathrm{erf}(b_{j}) + \frac{2b_{j}}{\pi}\!\sum\limits_{k=1}^{\frac{N}{2}}    
		\frac{w_{N,k}^{(H)}{\mathrm e}^{-b_j^2}}{b_j^2 +  (t_{N,k}^{(H)})^2} 
		- \frac{1}{\pi}\!\sum\limits_{k=1}^{\frac{N}{2}}    
		\frac{w_{N,k}^{(H)}{\mathrm e}^{-b_j^2}}{b_j- {\mathrm i}  t_{N,k}^{(H)}} {\mathrm e}^{-2(b_j- {\mathrm i} t_{N,k}^{(H)})(x-b_j) } 
		- \frac{1}{\pi}\!\sum\limits_{k=1}^{\frac{N}{2}}    
		\frac{w_{N,k}^{(H)}{\mathrm e}^{-b_j^2}}{b_j+ {\mathrm i}  t_{N,k}^{(H)}} {\mathrm e}^{-2(b_j+ {\mathrm i} t_{N,k}^{(H)})(x-b_j) }, 
		\end{align*}		
i.e., we indeed obtain an exponential sum of length $N+1$ that includes a constant term.
\end{remark}

\begin{Example}
	Taking for example $N=8$ and $\rho=4$, the estimate $(\ref{erf})$ yields
	\begin{equation*}
	\textstyle	|E(x)| \leq \frac{2}{\sqrt{\pi}}\frac{1}{\sqrt{34 {\mathrm e}}} {\mathrm e}^{-\frac{j^2}{{\mathrm e}}}  4^{-17} \approx  6.8320 \cdot 10^{-12}\cdot{\mathrm e}^{-\frac{j^2}{{\mathrm e}}}
	\end{equation*}
	for $x\in[\frac{j}{\sqrt{{\mathrm e}}},\frac{j+1}{\sqrt{{\mathrm e}}}]$. A plot of the error for $N=8$ and for $\rho=3$, $\rho=4$, and $\rho=6$ is illustrated in Figure \ref{figerf}.
	%	 and for $N=8$, $\rho=8$
	%	\begin{equation*}
		%		|E(x)| <  \frac{2}{\sqrt{\pi}}\frac{1}{\sqrt{34 {\mathrm e}}}{\mathrm e}^{-\frac{j^2}{4 {\mathrm e}}} 8^{-17} \approx  5.2124 \cdot 10^{-17}\cdot{\mathrm e}^{-\frac{j^2}{4 {\mathrm e}}}
		%	\end{equation*}
	%	for $x\in[\frac{j}{2\sqrt{{\mathrm e}}},\frac{j+1}{2\sqrt{{\mathrm e}}}]$.
\end{Example}
\begin{figure}[htb]
	\begin{center}
	\includegraphics[scale=0.28]{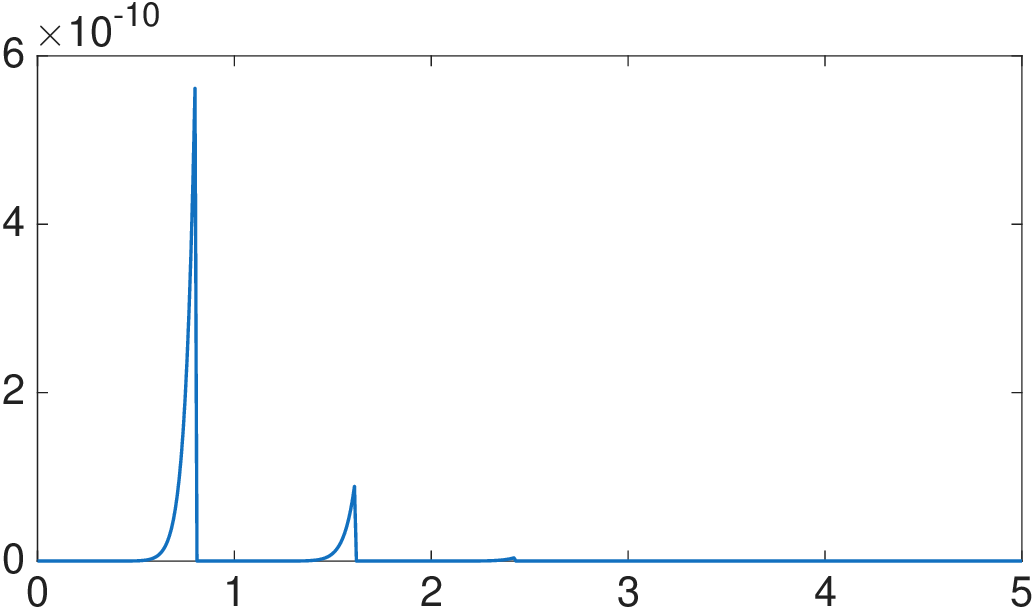}
	\includegraphics[scale=0.28]{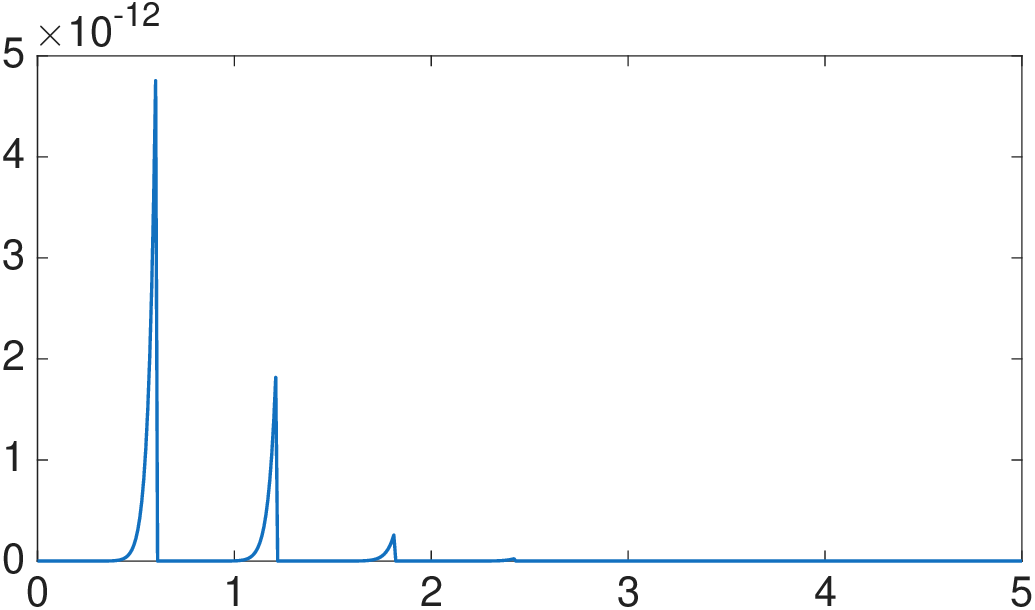}
	\includegraphics[scale=0.28]{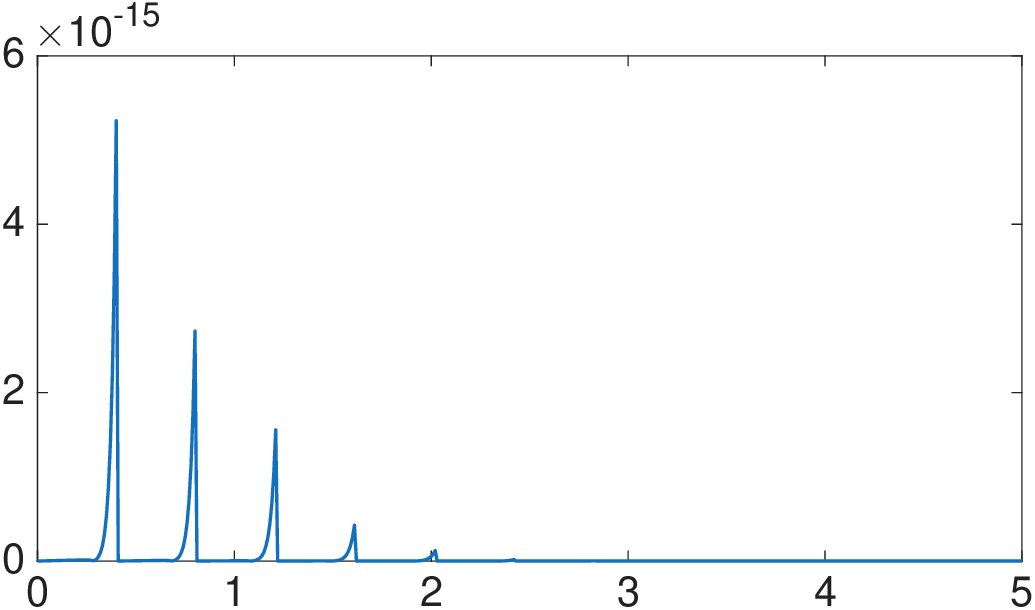}
	\end{center}
	\captionsetup{aboveskip=0pt, justification=justified, labelfont=small, labelsep=space, labelfont=bf}
	\caption{\small
		Illustration of the computational error $|E(x)|$ in (\ref{erf}) to approximate $ \mathrm {erf}(x)$ by an exponential sum of length $N+1=9$ on $[0,5)$ with $\rho=3$ (left), $\rho=4$ (middle), and $\rho=6$ (right).
%		Top: $N=6$  and $\rho=3$ (left) and $\rho=4$ (right).
%		Bottom: $N=8$  and $\rho=3$ (left) and $\rho=4$  (right).
	}
	\label{figerf}
\end{figure}

\section*{Acknowledgement}
The authors acknowledge funding from the European Union’s Horizon 2020 research and innovation programme under the Marie Sk\l{}odowska-Curie grant agreement No 101008231 (EXPOWER).

{\small
\bibliographystyle{abbrv}
\bibliography{gauss_bib.bib}
}
\end{document}